\documentclass[leqno,a4paper,12pt]{article}
\usepackage[applemac]{inputenc}
\usepackage{amsmath,amssymb,mathrsfs,amsthm}
\usepackage{enumerate}
\usepackage{fourier}

\numberwithin{equation}{section}

\newcommand{\dx}{\ensuremath{\hspace{0.1cm}\mathrm{d}x}}
\newcommand{\dt}{\ensuremath{\hspace{0.1cm}\mathrm{d}t}}
\newcommand{\ds}{\ensuremath{\hspace{0.1cm}\mathrm{d}s}}
\newcommand{\dW}{\ensuremath{\hspace{0.1cm}\mathrm{d}W}}
\newcommand{\ddW}{\ensuremath{\hspace{0.1cm}\mathrm{d}\tilde W}}

\newtheorem{lemma}{Lemma}
\newtheorem{definition}{Definition}
\newtheorem{proposition}{Proposition}
\newtheorem{corollary}{Corollary}
\newtheorem{theorem}{Theorem}
\newtheorem{remark}{Remark}

\providecommand{\msc}[1]{{\small \textit{Mathematics Subject Classification---} #1}}
\providecommand{\keywords}[1]{{\small \textit{Keywords---} #1}}

\title{A semi-discrete scheme for the stochastic Landau-Lifshitz equation}
\author{Fran\c cois Alouges\thanks{\textit{francois.alouges@polytechnique.edu}}, Anne De Bouard\thanks{\textit{debouard@cmap.polytechnique.fr}} and Antoine Hocquet\thanks{\textit{antoine.hocquet@cmap.polytechnique.fr}}\\{\small CMAP Ecole Polytechnique CNRS, Route de Saclay,}
\\{\small 91128 Palaiseau Cedex, FRANCE}}
\begin{document}
\maketitle

\begin{abstract}
We propose a new convergent time semi-discrete scheme for the \linebreak stochastic Landau-Lifshitz-Gilbert equation. The scheme is only linearly implicit and does not require the resolution
of a nonlinear problem at each time step. Using a martingale approach, we prove the convergence in law of the scheme up to a subsequence.
\end{abstract}
\msc{Primary 60H15 (35R60),35K55, 65M12; Secondary 82D45}

\section{Introduction}
\label{intro}
Ferromagnetic materials possess a spontaneous magnetization $m$ which evolution is classically modelized, when thermal fluctuations are negligible, according to the so-called Landau-Lifshitz equation
 (see e.g.\ \cite{BROWN1963,LANDAULIFSHITZ1935})
\begin{align}
 \begin{cases}
\label{LLG_formelle}
 \partial_t m =-\alpha m\times(m\times \mathrm{H_{eff}})+ m\times\mathrm{H_{eff}},\quad  \text{in }(0,T)\times D\,,\\
 \partial_n{m}(t,x)=0\quad \text{on }(0,T)\times\partial D\,,\\
 m(0,x):=m_0(x),\quad x\in D.\\
\end{cases}
\end{align}
In \eqref{LLG_formelle}, $D\subseteq\mathbb{R}^3$ is the domain occupied by the sample, $\alpha>0$ is a damping parameter.
The effective field $\displaystyle\mathrm{H_{eff}}:=-\frac{\partial \mathcal{E}(m)}{\partial m}$, where $\mathcal{E}(m)$ is the Brown energy of $m$ (see \cite{BROWN1978} for more details), which governs the dynamics, contains several terms according to different physical phenomena : exchange, anisotropy, stray field, external field, magnetostriction, etc. 
Notice that \eqref{LLG_formelle} preserves the local magnitude of the magnetization, namely, assuming $m_0(x)\in\mathbb{S}^2:=\big\{x\in\mathbb{R}^3,\;|x|=1\big\}$, we formally have
\begin{equation}\label{sphere}
 m(t,x)\in\mathbb{S}^2\,,\quad \forall(t,x)\in[0,T]\times D\,.
\end{equation} 
There is an abundant literature on the mathematical properties of \eqref{LLG_formelle}. We refer the reader to \cite{ALOUGESSOYEUR1991,carbou1997comportement,carbou2001regular,carbou2001regularsolutions,daquino2006midpoint,serpico2001numerical,serpico2003analytical,VISINTIN1985}, and references therein for the state-of-the-art of the analysis of \eqref{LLG_formelle}.

In \cite{BROWN1963}, thermal fluctuations are taken into account by adding to \eqref{LLG_formelle} a (stochastic) noise term $\xi(t,x)$. Following the presentation given in \cite{banas2013stochastic,BANASBRZEZNIAKPROHLNEKLYUDOV2013,banas2013computational,BRZEZNIAKGOLDYS2013},
and focusing on the case where only the exchange term is considered (i.e.\ $\mathrm{H_{eff}}=\Delta m$) we modify \eqref{LLG_formelle} to
\begin{align}
 \begin{cases}
\label{LLG_formelle_2}
 \partial_t m =-\alpha m\times(m\times \Delta m)+ m\times(\Delta m+\xi),\quad  \text{in }(0,T)\times D\,,\\
 \partial_n{m}(t,x)=0\quad \text{on }(0,T)\times\partial D\,,\\
 m(0,x):=m_0(x)\quad \forall x\in D\,.\\
\end{cases}
\end{align}
According to physicists, $\xi$ should be a Gaussian space-time white noise (see for instance, the review article \cite{Berkov} and references therein), that is uncorrelated in space. However, due to the lack of regularity of the space time white noise, equation \eqref{LLG_formelle_2} is not expected to possess a well defined solution in this case, and we therefore consider in this article a more regular noise in space. Namely, let $W$ be a cylindrical Wiener process that is given by the expression
\begin{equation*}
 W(t)=\sum_{i\in\mathbb{N}}\beta_i(t)e_i \,,
\end{equation*}
where $(e_i)_{i\in\mathbb{N}}$ denotes a complete orthonormal system of $L^2(D)^3$, and $(\beta_i)_{i\in\mathbb{N}}$ stands for a sequence of
real valued and independent brownian motions.
Then, writing for each $i\in\mathbb{N}$, $G_i:=Ge_i$,  we set
\begin{equation*}
 \xi(t,x)=G\dot W=\sum_{i\in\mathbb{N}}\dot\beta_i(t)G_i(x)\,,
\end{equation*} 
where $G$ is a given Hilbert-Schmidt operator from the space $L^2(D)^3$ into the space $H^2(D)^3$.

As long as we have not specified the choice of the stochastic integral, \eqref{LLG_formelle_2} may have different meanings \cite{van1981ito}.
It is well-known (see for instance \cite{REZNIKOFF2004}) that in order to satisfy the geometrical constraint \eqref{sphere}, the product with the noise term must be understood in the Stratonovich sense, which leads to the stochastic Landau-Lifshitz equation (SLL), where, for simplicity, we have set the parameter $\alpha$
to one :
\begin{equation}\label{LLG_Strato}
\mathrm{d}m=\big(-m\times(m\times\Delta m)+m\times\Delta m\big) \dt+m\times\circ (G \dW),\text{ for } (t,x)\in[0,T]\times D.
\end{equation}
Here we have denoted by ``$\circ\,\mathrm{d}$'' the Stratonovich differential, and we will denote by ``$\mathrm{d}$'' the Ito differential.
We set the initial condition $m(0,x)=m_0(x)$, for any $x\in D$.

In order to work with a non-anticipative integral, we change \eqref{LLG_Strato} to its Ito form.
Using the formal relation between the Stratonovich and Ito differentials,
a corresponding Ito formulation of \eqref{LLG_Strato} is obtained by adding a correction term to the drift of \eqref{LLG_Strato}. This term is what we may call in the sequel ``the Ito correction''. In this sense, the noise term can be rewritten as follows (see e.g.\  \cite{banas2013stochastic,BANASBRZEZNIAKPROHLNEKLYUDOV2013,BRZEZNIAKGOLDYS2013})~:
\begin{equation*}
 m\times\circ(G\dW)=m\times(G\dW)+\frac{1}{2}\sum_{i\in\mathbb{N}}(m\times G_i)\times G_i \dt.
\end{equation*}
Noticing furthermore that $m\times(m\times\Delta m)$ formally equals $-\Delta m-m|\nabla m|^2$, we rewrite equation \eqref{LLG_Strato} as  
\begin{equation}\label{LLG_Ito}
\mathrm{d}m =\big(\Delta m+m|\nabla m|^2+m\times\Delta m+\frac{1}{2}\sum_{i\in\mathbb{N}} (m\times G_i)\times G_i\big)\dt  +m\times (G \dW )\,.
\end{equation}
We refer to \cite{banas2013stochastic} for a review of the existing results on equation \eqref{LLG_Ito}.

The so-called Gilbert form (SLLG) is (still formally) obtained by applying the operator $(\mathrm{Id}-m\times\cdot)$ to the previous equation~:
\begin{eqnarray}
\label{LLG_GILBERT_form}
 \mathrm{d}m -m\times\mathrm{d}m&=&\big[2\big(\Delta m+m|\nabla m|^2\big)+\frac12\sum_{i\in\mathbb{N}}
 (\mathrm{Id}-m\times) \big((m\times G_i)\times G_i\big)\big]\dt \nonumber \\
& & +(\mathrm{Id}-m\times)\big(m\times (G \dW )\big)\,.
 \end{eqnarray}
Equivalence between \eqref{LLG_Ito} and \eqref{LLG_GILBERT_form} is not clearly stated in the litterature and we therefore establish it in Remark \ref{equivalence_LLG}.

Developing numerical schemes for the simulation of LLG plays a prominent role in the modeling of ferromagnetic materials. We refer the reader to \cite{cimrak2005error,cimrak2007survey,cimrak2004iterative} for an overview of the literature on the subject. However, reliable schemes for the simulation of (SLL, SLLG) remain very few. Probably the first scheme for which convergence can be proved is given in \cite{BANASBRZEZNIAKPROHLNEKLYUDOV2013} and is based on a Crank-Nicolson-type time-marching evolution which relies on a nonlinear iteration solved by a fixed point method. On the other hand, there has been in the past recent years an intensive development of a new class of numerical methods for LLG, based on a linear iteration, and for which unconditional convergence and stability can be shown \cite{alouges2008new,alouges2006convergence,bartels2008numerical,kritsikis2013beyond}.
 The aim of this paper is to extend the ideas developed there and generalize the scheme in order to take into account the stochastic term. We only consider a time semi-discrete approximation of (SLL, SLLG) for which we show the unconditional convergence when the time step tends to $0$.
Proving the convergence of the fully discrete approximation (using a finite element method in space) would not cause any major difficulty (see \cite{alouges2006convergence} for details).

We think that the methodology we develop can be generalized for stochastic differential or partial differential equations with a geometrical constraint. It is definitely different from -- though related to -- the approach of \cite{lelievre2008analyse} (see Remark \ref{lelievre} below). 

\paragraph{Notation.}
Throughout this paper, we assume that $T>0$ is a given constant and
$\left(\Omega,\mathcal{F},\mathbb{P},(\mathcal{F}_t)_{t\in[0,T]},(W_t)_{t\in[0,T]}\right)$
is a stochastic basis, that is $(\Omega,\mathcal{F},\mathbb{P})$ is a probability space, $(\mathcal{F}_t)_{t\in[0,T]}$ is
a filtration and $(W_t)_{t\in[0,T]}$ a cylindrical Wiener process adapted to $(\mathcal{F}_t)$.
The domain $D\subset\mathbb{R}^3$ is supposed to be bounded ; we denote by $a\cdot b$, where $a,b\in\mathbb{R}^3$ (resp.\ $\mathbb{R}^{3\times3}$), the standard scalar product in $\mathbb{R}^3$ (resp.\ $\mathbb{R}^{3\times3}$), and by $|\cdot|$ the associated euclidean norm. Norms in Banach spaces are in turn denoted by $\|\cdot\|$. In particular, the notation $\|f\|_{p,x}$ will be used to designate indifferently the $L^p(D)^3$ or $L^p(D)^{3\times3}$ norm. The inner product in the space $L^2(D)^3$ (respectively $L^2(D)^{3\times3}$) of square integrable functions with values in $\mathbb{R}^3$ (respectively $\mathbb{R}^{3\times3}$) is denoted by $(\cdot,\cdot)_{2,x}$, namely
\begin{equation*}
\forall f,g\in L^2(D)^3\,,\qquad
(f,g)_{2,x}:=\int_D f(x)\cdot g(x)\dx\,,
\end{equation*}
and
\begin{equation*}
\forall F,G\in L^2(D)^{3\times3}\,,\qquad
(F,G)_{2,x}:=\int_D F(x)\cdot G(x)\dx\,.
\end{equation*}
The notation $\mathcal{C}\big([0,T];X\big)$, where $X$ is a Banach space, is used to denote the space of continuous functions from $[0,T]$ into $X$.
Classical Sobolev spaces of $\mathbb{R}^3$-valued functions are denoted by $W^{\alpha,p}(D)$, $\alpha\in\mathbb{R}$, or $H^\alpha(D)$ when $p=2$,  (see e.g.\ \cite{ADAMSFOURNIER2003}). Finally, the norm of a Hilbert-Schmidt operator from $L^2(D)$ into $H^{\alpha}(D)$ is denoted by $\|\cdot\|_{2,\alpha}$.
For a given number of time intervals $N\in\mathbb{N}^*$, we define the time step ${\Delta t}:=\frac{T}{N}$, and ${\Delta{W}_N^n}:=W((n+1){\Delta t})-W(n{\Delta t}) $, for any $n$ with $0\leq n\leq N-1$.
Therefore $G{\Delta{W}_N^n}$ is a gaussian random variable on $L^2(D)^3$ with covariance
operator $(\Delta t)GG^*$.

\section{Main result}
\label{sec:main_result}

Our purpose is to analyse a semi-implicit scheme with parameter $\theta\in(\frac{1}{2},1]$. Unlike the approach used in \cite{BANASBRZEZNIAKPROHLNEKLYUDOV2013}, we use the Gilbert form of the equation, i.e.\ equation \eqref{LLG_GILBERT_form}. 
This approach allows us to overcome the difficulty of solving a nonlinear system at each step of the algorithm. Given the data $m_N^n$, where $N$ is the number of time steps, and $n\in\{0,\dots,N-1\}$, which is an approximation of $m(n\Delta t)$, the unknown $v_N^n$, namely the tangential increment of $m_N^n$, can be found simply by solving a \emph{linear} system. Indeed, following an idea of \cite{alouges2006convergence}, one may search $v_N^n$ in the subset of $H^1(D)^3$ whose elements are almost everywhere orthogonal to $m_N^n$, so that the non linear term in $m_N^n\times(m_N^n\times\Delta m_N^n)=-\Delta m_N^n - m_N^n|\nabla m_N^n|^2$ vanishes when testing against functions that also satisfy this constraint. Roughly speaking, the test functions in the following formulation \eqref{pv} "only see" the part of $m_N^{n+1}-m_N^n$ which is orthogonal to $m_N^n$, but this is however sufficient, as shown in Section 6.

Let us now describe the scheme rigorously.
We fix the parameter 
\begin{equation}\label{theta_hyp}
 \theta\in(\frac{1}{2},1]\,,
\end{equation}
and assume that the operator $G:L^2(D)^3\rightarrow H^2(D)^3$ satisfies
\begin{equation}\label{G_hypothesis}
\|G\|_{2,2}^2=\sum_{i\in\mathbb{N}}\|G_i\|_{H^2(D)^3}^2<\infty\,.
\end{equation}

Our algorithm reads as follows, for a given integer $N>0$ :

\paragraph{Algorithm (*)~:}
{\itshape Fix
\begin{equation}\label{algo_1}
m_N^0:=m_0\in H^1(D)^3\,, 
\end{equation}
and for any $n\in\{0,\dots,N-1\}$, suppose that the random variable  $m_N^n(\omega, \cdot)\linebreak\in H^1(D)^3$ is known. 
Let $v_N^n(\omega,\cdot)$ be the unique solution in the space \begin{equation*}
\mathbb{W}_{N,n}(\omega)~:=\Big\{ \psi\in H^1(D)^3, \;   \forall x\in D, \;  \psi(x)\perp{m}_N^n(\omega,x) \Big\}\,,
\end{equation*}  of the variational problem~: $\forall \varphi \in \mathbb{W}_{N,n}(\omega)$,
\begin{flalign}
 \label{pv}
\Big(v_N^n-  {m}_N^n&\times {v}_N^n ,\varphi\Big)_{2,x} + 2\theta{\Delta t}\Big(\nabla  {v}_N^n,\nabla\varphi\Big)_{2,x}&\nonumber \\
=&-2{\Delta t}\Big(\nabla   {m}_N^n,\nabla\varphi\Big)_{2,x}+ \Big((\mathrm{Id}- {m}_N^n\times)\big(  {m}_N^n\times G{\Delta{W}_N^n}\big),\varphi\Big)_{2,x}&\nonumber\\
& +\frac{{\Delta t}}{2}\sum_{i\in\mathbb{N}}\Big((\mathrm{Id}-  {m}_N^n\times)\big((  {m}_N^n\times G_i)\times G_i\big),\varphi\Big)_{2,x}.&
\end{flalign}
Then, we set, for all $(\omega,x)\in\Omega\times D$,
\begin{equation}\label{pv_renorm}
  {m}_N^{n+1}(\omega,x)=\frac{  {m}_N^n(\omega,x)+ {v}_N^n(\omega,x)}{|  {m}_N^n(\omega,x)+ {v}_N^n(\omega,x)|}\,.
\end{equation}
}

Note that the formulation \eqref{pv} is a $\theta$-scheme applied to the variational formulation of equation \eqref{LLG_GILBERT_form} (see \cite{alouges2006convergence}).
One has $  {m}_N^n\in H^1(D)^3$ a.s., for any $n\in\{0,\dots,N\}$ and $( {m}_N^n)_{0\leq n\leq N}$ is adapted to the filtration $(\mathbb{F}_N^n)_{0\leq n\leq N}$ defined by
\begin{equation}\label{discrete_filtration}
\mathbb{F}_N^n:=\sigma\big\{GW(k\Delta t)\,,0\leq k\leq n\big\}\,.
\end{equation}
Indeed, it is not difficult to prove that under the above assumptions, problem \eqref{pv} admits a unique solution $ {v}_N^n(\omega,\cdot)\in\mathbb{W}_{N,n}(\omega)$ (see \cite{alouges2006convergence} for a proof in the deterministic case). The noise and correction terms do not alter the hypotheses of the Lax-Milgram theorem. Moreover, this solution depends continuously in $H^1(D)^3$ on the two arguments $({m}_N^n\,,G{\Delta{W}_N^n})$, for the $H^1(D)^3\times L^2(D)^3$ topology. It implies in particular that the law of $ {v}_N^n$ on $H^1(D)^3$ only depends on the law of $( {m}_N^n\,,G{\Delta{W}_N^n})$ on $H^1(D)^3\times L^2(D)^3$.

\begin{remark}
\label{lelievre}
As mentioned before, the approach here is different from the one in \cite{lelievre2008analyse}, where the approximation of solutions of some Stratonovich stochastic differential equation with values in a manifold is considered. Indeed, in \cite{lelievre2008analyse}, the scheme consists in using the explicit Euler scheme (which approximates the Ito equation) on one time step, and then projecting the solution on the manifold.
Here, we do not approximate the It\^o equation, since part of the It\^o correction is put in the increment. We will see that the projection on the manifold (the sphere here) brings the remaining part of the It\^o correction..
\end{remark}

We now give the definition of the martingale solutions of equation \eqref{LLG_Ito} that we consider here, which is similar to
the one in \cite{BANASBRZEZNIAKPROHLNEKLYUDOV2013,BRZEZNIAKGOLDYS2013}).

\begin{definition}[Martingale solution]
 Given $T>0$, a {martingale solution} on $[0,T]$ of \eqref{LLG_Ito} is given by a filtered probability space $\big(\tilde\Omega,\mathcal{\tilde F},\mathbb{\tilde P},(\mathcal{\tilde{F}}_t)\big)$, together with \linebreak$H^1(D)^3$-valued, progressively measurable processes $G\tilde W$ and $\tilde m$ defined on this space, where
 $G\tilde W$ is a Wiener process on $\big(\tilde\Omega,\mathcal{\tilde F},\mathbb{\tilde P},(\mathcal{\tilde{F}}_t)\big)$, with covariance operator $GG^*$, and $\tilde m$ satisfies the following assumptions :
\begin{enumerate}
 \item $\tilde m(\omega,\cdot) \in C([0,T];L^2(D)^3)$, a.s.
 \item for any $t\in[0,T]$, $\tilde m(t)$ belongs to $H^1(D)^3$, and the random variable $\Delta \tilde m + \tilde m|\nabla \tilde m|^2 + \tilde m\times\Delta \tilde m+\frac{1}{2} \sum_{i\in\mathbb{N}}( \tilde m\times G_i)\times G_i$ takes its values in the space $L^1(0,T;\linebreak L^2(D)^3)$, a.s.
\item $\tilde m$ satisfies \eqref{LLG_Ito}; more precisely,
\begin{eqnarray*}
  \tilde m(t)&=&{m}_0+\int_0^t\Big(\Delta \tilde m(s) + \tilde m(s)|\nabla \tilde m(s)|^2 + \tilde m(s)\times\Delta \tilde m(s)\\
  & & +\frac{1}{2} \sum_{i\in\mathbb{N}}( \tilde m(s)\times G_i)\times G_i\Big)\ds +\int_0^t \tilde m(s)\times (G\ddW(s))
\end{eqnarray*}
where the first integral is the Bochner integral in $L^2(D)^3$, and the second is the Ito integral of a predictable $H^1(D)^3$-valued process.
\item $|\tilde m(\omega,t,x)|=1\,$ for almost every $(\omega,t,x)\in \tilde\Omega\times[0,T]\times D$.
\end{enumerate}
\end{definition}

\begin{remark}\label{equivalence_LLG}
Note that if $\tilde m$ is a martingale solution of \eqref{LLG_Ito}, then we can rewrite the stochastic integral of the predictable process $s\mapsto\tilde m(s)\times$ with respect to the semimartingale $\tilde m$ as
\begin{multline*}
\int_0^t \tilde m(s)\times \mathrm{d}\tilde m(s)\\=\int_0^t \tilde m(s)\times \mathrm{d}\left(\int_0^sI(\sigma)\mathrm{d}\sigma\right)+\int_0^t \tilde m(s)\times \mathrm{d}\left(\int_0^s \tilde m(\sigma)\times G\mathrm{d}\tilde W(\sigma)\right)
\end{multline*}
where $I$ is given by
 $\forall s\in[0,T]$:
\begin{eqnarray*}
I(s)&=& \Delta\tilde m(s) + {\tilde m(s)}|\nabla {\tilde m(s)}|^2 +\tilde m(s)\times\Delta \tilde m(s)\\
&&+\frac{1}{2}\sum_{i\in\mathbb{N}}({\tilde m(s)}\times{G_i})\times{G_i} \in L^2(\Omega\times D)^3\,.
\end{eqnarray*}
It then follows from classical properties of stochastic integrals with respect to semimartingales that
$$
\int_0^t (\mathrm{Id}-\tilde m(s)\times) \mathrm{d}\tilde m(s)=\int_0^t (\mathrm{Id}-\tilde m(s)\times) I(s)\ds+\int_0^t (\mathrm{Id}-\tilde m(s)\times)G\mathrm{d}\tilde W(\sigma)\,,
$$
and since for almost all $\omega,t,x$, $|\tilde m(\omega,t,x)|=1$, $\tilde m$ is also a solution to \eqref{LLG_GILBERT_form}. Thus \eqref{LLG_Ito} and \eqref{LLG_GILBERT_form} are in fact equivalent.
\end{remark}

Our main result is then given by the following theorem, and says that, up to a subsequence, the discrete solution $m_N$ of the algorithm (*)
converges in law to a martingale solution of equation \eqref{LLG_Ito}.

\begin{theorem}[Convergence of the algorithm]\label{theo_main_result}
For every $N\in\mathbb{N}^*$, we define the progressively measurable $H^1(D)^3$-valued process $m_N$ by:
\begin{equation*}
  m_N(t):=  {m}_N^n\quad \text{if } t\in[n{\Delta t},(n+1){\Delta t}).
\end{equation*}
There exists a {martingale solution} of \eqref{LLG_Ito}
$\big(\tilde\Omega, \mathcal{\tilde F}, \mathbb{\tilde P},(\mathcal{\tilde F}_t)_{t\in[0,T]}, (\tilde W_t)_{t\in[0,T]},\tilde m\big)\,,$ and a sequence $(\tilde m_{N})_{N\in\mathbb{N}^*}$ of random processes defined on $\tilde\Omega$,
with the same law as $m_N$, so that up to a subsequence, the following convergence holds :
\begin{equation*}
 \tilde m_N\underset{N\to\infty}{\longrightarrow}\tilde m,\quad \text{in }L^2(\tilde\Omega\times[0,T]\times D)^3.
\end{equation*}
\end{theorem}

In order to prove Theorem \ref{theo_main_result}, we proceed in several steps.
In Section 3, we establish uniform estimates for several processes. We first establish a uniform bound on the Dirichlet energy of $m_N$, thanks to the variational formulation, and an appropriate choice of the test function. 
Section 4 is devoted to the proof of the tightness of the sequence $(m_N)$ on the space $L^2([0,T]\times D)^3$. 
After a change of probability space, we can assume that there exists an almost sure limit $m$ of $(m_N)$, that is
\begin{equation*}
  m_N\underset{N\to\infty}{\longrightarrow}m\quad \text{a.s.\ in } L^2([0,T]\times D)^3.
\end{equation*}
Then, setting
\begin{equation}\label{decomposition}
  m_N(t) = {m}_0 + F_N(t)+ X_N(t) , 
\end{equation} 
where $(X_N(n\Delta t))_{0\leq n\leq N-1}$ defines an $L^2(D)^3$-valued discrete parameter martingale, with respect to the filtration $(\mathbb{F}_N^n)_{0\leq n\leq N-1}$ (see \eqref{discrete_filtration}), and $F_N(t)$ is, for each $N$, a deterministic function of $\left.m_N\right|_{[0,t]}$,
we use \eqref{pv} and the previous energy estimates to identify $F_N(t)$ and its limit up to a subsequence.
In section 5, we show that, still up to a subsequence, $X_N(t)$ converges to a limit $X(t)$ which is a square-integrable continuous martingale with an explicit quadratic variation. The martingale representation theorem allows us to conclude~:
there exists a new filtered probability space for which the limit of the martingale part is a stochastic integral with respect to a Wiener process $G\tilde W$
with covariance operator $GG^*$.
Finally we use the limit of \eqref{decomposition} and this latter stochastic integral in order to identify the equation satisfied by $m(t)$.
The explicit form of the limit $F(t)$ of $F_N(t)$ as $N\to\infty$ is the Bochner integral of the $L^2(D)^3$-valued process $t'\mapsto\Delta m(t') +m(t')|\nabla m(t')|^2 + m(t')\times\Delta m(t') +\frac{1}{2}\sum_{i\in\mathbb{N}}(m(t')\times G_i)\times G_i$ on the time interval $[0,t]$, which allows us to conclude.

\section{Energy estimates}
\label{sec:energy_estimates}

Fix $N>0$, and set $ {m}_N^0= {m}_0$.
Let $( {m}_N^n)_{0\leq n\leq N}$ and $( {v}_N^n)_{0\leq n\leq N}$ be given by the algorithm (*).
In all what follows, we write 
\begin{equation}\label{nota_AN}
  {A}_N^n:=  {m}_N^n\times(G\Delta W_N^n)\,.
\end{equation} 
This term corresponds to the noise term which is added at each step of the algorithm.
Thanks to the Gaussian properties of $G{\Delta{W}_N^n}$, the fact that $\| {m}_N^n\|_{L^\infty(D)^3}\leq 1$, and the Sobolev embeddings, we have the following obvious, but useful estimates~: for all $n\in\{0,\dots, N\}$,
\begin{equation}\label{estim_A_2}
 \mathbb{E}\big[\| {A}_N^n\|_{2,x}^2\big]\leq{\Delta t}\|G\|_{2,0}^2\,,
\end{equation}
and
\begin{equation}\label{estim_A_4}
 \mathbb{E}\big[\| {A}_N^n\|^4_{L^4}\big]\leq C({\Delta t})^2\|G\|^4_{2,1}.
\end{equation}

\begin{proposition}\label{pro_estim_discretes} There exists a constant $C=C(T, {m}_0,\|G\|_{2,2})$, so that for all $N\in \mathbb{N}^*$,
\begin{equation}\label{estim_u}
 \max_{n=0\dots N}\mathbb{E}\left[\|\nabla   {m}_N^n\|_{2,x}^2\right]\leq C,
\end{equation}
\begin{equation}\label{estim_w}
 \mathbb{E}\left[\sum_{n=0}^{N-1}\| {v}_N^n- {A}_N^n\|_{2,x}^2\right]\leq C{\Delta t},
\end{equation}
\begin{equation}\label{estim_v}
 \mathbb{E}\left[\sum_{n=0}^{N-1}\| {v}_N^n\|_{2,x}^2\right]\leq C,
\end{equation}
\begin{equation}\label{estim_nabla_v}
 \mathbb{E}\left[\sum_{n=0}^{N-1}\|\nabla  {v}_N^n\|_{2,x}^2\right]\leq C\,.
\end{equation}
\end{proposition}

The proof of Proposition \ref{pro_estim_discretes} uses the following remark, together with the estimate of Lemma \ref{lem_nabla_u_v} below, whose proof is postponed to the end of section $3$.
\begin{remark}\label{rema_H_integralenergy_dicreasing}
The renormalization stage decreases the Dirichlet energy. Indeed, it was shown in \cite{ALOUGES1997} that for any map
$\psi\in H^1(D)^3$, such that 
a.e. in $D$, $|\psi(x)|\geq~1$,
one has
\begin{equation}\label{ineg_renorm}
 \int_D\left|\nabla\left(\frac{\psi(x)}{|\psi(x)|}\right)\right|^2\dx\leq \int_D\left|\nabla\left(\psi(x)\right)\right|^2\dx\,.
\end{equation}
\end{remark}

\begin{lemma}\label{lem_nabla_u_v}
For all $\epsilon$ with $0<\epsilon<2\theta-1$, there exists $C=C(\epsilon, \|G\|_{2,2},T)>0$ such that for all $N\in\mathbb{N}^*$, 
and $n=0,\dots,N-1$~:
\begin{multline}\label{Gronwall}
\mathbb{E}\big[|\nabla {m}_N^{n+1}\|_{2,x}^2\big] +\frac{(1-\epsilon)}{{\Delta t}}\mathbb{E}\big[\| {v}_N^n- {A}_N^n\|_{2,x}^2\big] +(2\theta-1-\epsilon)\mathbb{E}\big[\|\nabla  {v}_N^n\|_{2,x}^2\big]\\
\leq (1+C{\Delta t})\mathbb{E}\big[\|\nabla   {m}_N^n\|_{2,x}^2\big]+C{\Delta t}\,.
\end{multline}
\end{lemma}
Let us now prove Proposition \ref{pro_estim_discretes} with the help of Lemma \ref{lem_nabla_u_v}.
\begin{proof}[Proof of Proposition \ref{pro_estim_discretes}]
In the sequel, we fix $\epsilon\in(0,2\theta-1)$.
We first prove \eqref{estim_u}.
We deduce from \eqref{Gronwall} that
for all $n=0\dots N$
\begin{equation*}
 \mathbb{E}\big[\|\nabla {m}_N^{n+1}\|_{2,x}^2\big]\leq (1+C{\Delta t})\mathbb{E}\big[\|\nabla {m}_N^n\|_{2,x}^2\big]+C{\Delta t}\,.
\end{equation*}
We then apply the discrete Gronwall lemma.
There exists $C=C(\|G\|_{2,2}, T)>0$ such that
for all $n=0\dots N$, 
\begin{equation*}
 \mathbb{E}[\|\nabla   {m}_N^n\|_{2,x}^2]\leq C(1+\mathbb{E}[\|\nabla m_0\|_{2,x}^2]),
\end{equation*}
and \eqref{estim_u} is proved.

We now turn to the proof of \eqref{estim_w}-\eqref{estim_nabla_v}.
We note that \eqref{Gronwall} implies in particular
\begin{multline*}
\mathbb{E}\big[|\nabla {m}_N^{n+1}\|_{2,x}^2\big] -\mathbb{E}\left[\|\nabla   {m}_N^n\|_{2,x}^2\right]
+\frac{(1-\epsilon)}{{\Delta t}}\mathbb{E}\big[\| {v}_N^n- {A}_N^n\|_{2,x}^2\big] \\
+(2\theta-1-\epsilon)\mathbb{E}\big[\|\nabla  {v}_N^n\|_{2,x}^2\big]\leq C{\Delta t}\mathbb{E}\big[\|\nabla   {m}_N^n\|_{2,x}^2\big]+C{\Delta t}\,.
\end{multline*}
By summing these inequalities for $n=0\dots N-1$, we obtain
\begin{eqnarray*}
 \mathbb{E}\left[\|\nabla m_N^{N}\|_{2,x}^2\right]-\mathbb{E}\left[\|\nabla m_N^{0}\|_{2,x}^2\right] &+&\frac{(1-\epsilon)}{{\Delta t}} \sum_{n=0}^{N-1} \mathbb{E}\big[\| {v}_N^n- {A}_N^n\|_{2,x}^2\big]\\
+(2\theta-1-\epsilon)  \sum_{n=0}^{N-1} \mathbb{E}\big[\|\nabla  {v}_N^n\|_{2,x}^2\big] &\leq&  \sum_{n=0}^{N-1}C{\Delta t}\mathbb{E}\big[\|\nabla {m}_N^n\|_{2,x}^2\big]+\sum_{n=0}^{N-1}C{\Delta t}\\[0.4cm]
&\leq& C(\|G\|_{2,2},T,m_0)\,,
\end{eqnarray*}
thanks to \eqref{estim_u}.
This implies that~:
\begin{flalign*}
\frac{(1-\epsilon)}{{\Delta t}}\sum_{n=0}^{N-1}\mathbb{E}\big[\| {v}_N^n- & A_N^n\|_{2,x}^2\big] +(2\theta-1-\epsilon)  \sum_{n=0}^{N-1} \mathbb{E}\big[\|\nabla  {v}_N^n\|_{2,x}^2\big]& \\
&\leq C(\|G\|_{2,2},T,m_0)-\mathbb{E}\left[\|\nabla  {m}_N^N\|_{2,x}^2\right]+\mathbb{E}\left[\|\nabla  {m}_N^0\|_{2,x}^2\right]&\\
&\leq  C'(\|G\|_{2,2},T,m_0)\,.&
\end{flalign*}
Thus, \eqref{estim_w} and \eqref{estim_nabla_v} follow.
Finally, we may deduce \eqref{estim_v} from \eqref{estim_w} and \eqref{estim_A_2}.
This ends the proof of Proposition \ref{pro_estim_discretes}.
\end{proof}

We now turn to the proof of the lemma.
\begin{proof}[Proof of Lemma \ref{lem_nabla_u_v}.]
Let $0\leq n \leq N-1$. Since, by definition of the variational problem \eqref{pv}, $m_N^n(x)\cdot v_N^n(x)=0$, almost everywhere, and almost surely, it follows that for a.e. $x\in D$, a.s.
$$|  {m}_N^n(x)+ {v}_N^n(x)|=\sqrt{1+| {v}_N^n(x)|^2}\geq 1,$$
and thanks to Remark \ref{rema_H_integralenergy_dicreasing},
one has a.s.
\begin{eqnarray*}
\|\nabla   {m}_N^{n+1}\|_{2,x}^2 &=& \int_D \left|\nabla \left(\frac{  {m}_N^n+ {v}_N^n}{|  {m}_N^n+ {v}_N^n|}\right)\right|^2\dx\\
&\leq& \int_D|\nabla {m}_N^n+\nabla  {v}_N^n|^2\dx\,.
\end{eqnarray*}
Then, by expanding the right hand side of this inequality:
\begin{equation}\label{ineg}
\|\nabla   {m}_N^{n+1}\|_{2,x}^2 \leq \|\nabla {m}_N^n\|_{2,x}^2+2\big(\nabla {m}_N^n, \nabla  {v}_N^n\big)_{2,x}+\|\nabla  {v}_N^n\|_{2,x}^2\,.
\end{equation}
To find an expression of $2\big(\nabla {m}_N^n, \nabla  {v}_N^n\big)_{2,x}$, we use \eqref{pv} with the test function $$\varphi:= {v}_N^n- {A}_N^n\, \in\mathbb{W}_{N,n}\,.$$
Then, observing that for any $x\in D$
\begin{flalign*}
( {v}_N^n- {m}_N^n&\times v_N^n)(x)\cdot ( {v}_N^n- {A}_N^n)(x)&\\
-&( \mathrm{Id}- {m}_N^n\times)(m_n^n\times G\Delta W_N^n)(x) \cdot ( {v}_N^n- {A}_N^n)(x)&\\
&\hspace{2cm}= \big(( \mathrm{Id}- {m}_N^n\times)( {v}_N^n- {A}_N^n)(x)\big)\cdot \big( {v}_N^n(x)- {A}_N^n(x)\big)&\\
&\hspace{2cm}=| {v}_N^n(x)- {A}_N^n(x)|^2,&
\end{flalign*}  one has
\begin{multline}\label{rel_u_v}
2\big(\nabla {m}_N^n,\nabla  {v}_N^n\big)_{2,x}= 
-\frac{1}{{\Delta t}}\| {v}_N^n- {A}_N^n\|_{2,x}^2-2\theta\|\nabla  {v}_N^n\|_{2,x}^2+2\theta\big(\nabla  {v}_N^n,\nabla  {A}_N^n\big)_{2,x}\\
+2\big(\nabla {m}_N^n,\nabla  {A}_N^n\big)_{2,x} +\frac12 \sum_{i\in\mathbb{N}}\big( (\mathrm{Id}-  {m}_N^n\times)((m_N^n\times G_i )\times G_i), {v}_N^n- {A}_N^n\big)_{2,x}\,,
\end{multline}
which, using \eqref{ineg} and taking the expectation yields to
\begin{multline}\label{relation_G}
\mathbb{E}\big[\|\nabla {m}_N^{n+1}\|_{2,x}^2\big] +\frac{1}{{\Delta t}}\mathbb{E}\big[\| {v}_N^n- {A}_N^n\|_{2,x}^2\big]+(2\theta-1)\mathbb{E}\left[\|\nabla  {v}_N^n\|_{2,x}^2\right] \\
\leq \mathbb{E}\big[\|\nabla {m}_N^n\|_{2,x}^2\big] +2\theta\mathbb{E}\big[\big(\nabla  {v}_N^n,\nabla {A}_N^n\big)_{2,x}\big]+2\mathbb{E}\big[\big(\nabla {m}_N^n,\nabla  {A}_N^n\big)_{2,x}\big]\\
 +\frac{1}{2}\mathbb{E}\big[\sum_{i\in\mathbb{N}} \big( (\mathrm{Id}-  {m}_N^n\times)((m_N^n\times G_i )\times G_i),  {v}_N^n- {A}_N^n\big)_{2,x}\big]\,,
\end{multline}
all terms on the left hand side being non negative, due to $\theta \in (\frac12,1]$.

Now, since ${m}_N^n$ and ${G\Delta{W}_N^n}$ are independent, and $\mathbb{E}[G{\Delta{W}_N^n}]=0$,
we have
\begin{equation}
\label{nablaMA}
\mathbb{E}\left[\big(\nabla   {m}_N^n,\nabla  {A}_N^n\big)_{2,x}\right]=0\,.
\end{equation}
Moreover, by \eqref{nota_AN} and the Sobolev embedding $H^2(D)^3\subset L^{\infty}(D)^3$,
\begin{eqnarray}
\mathbb{E} \left[\|\nabla A_N^n\|_{2,x}^2\right] &\leq& 2\left(\mathbb{E}\|\nabla m_N^n\|_{2,x}^2\mathbb{E}\left[\|G\Delta W_N^n \|^2_{\infty,x}\right] +\mathbb{E} \left[\|\nabla (G\Delta W_N^n)\|_{2,x}^2\right]\right)\nonumber\\
&\leq &C\Delta t \|G\|_{2,2}^2 \big(\mathbb{E}\left[\|\nabla   {m}_N^n\|_{2,x}^2\right]+1\big)\label{estim_nabla_A}.
\end{eqnarray}
Therefore
\begin{equation}
\label{nablaVA}
 \mathbb{E}\left[\big(\nabla {v}_N^n,\nabla  {A}_N^n\big)_{2,x}\right]
 \leq \frac{\epsilon}{2}\mathbb{E}\left[\|\nabla  {v}_N^n\|_{2,x}^2\right]+\frac{{C\Delta t}}{2\epsilon}\|G\|_{2,2}^2
 \big(\mathbb{E}\left[\|\nabla   {m}_N^n\|_{2,x}^2\right]+1\big)\,.
\end{equation}
Similarly, one has
\begin{multline}
\label{Immx}
 \mathbb{E}\left[\Big(\sum_{i\in\mathbb{N}}(\mathrm {Id}- {m}_N^n)\times\big((  {m}_N^n\times G_i)\times G_i)\Big), {v}_N^n- {A}_N^n\Big)_{2,x}\right]\\
\leq \frac{{\Delta t}}{2\epsilon}\|G\|_{2,2}^4 + \frac{2\epsilon}{{\Delta t}}\mathbb{E}\left[\| {v}_N^n- {A}_N^n\|_{2,x}^2\right]\,.
\end{multline}
Using \eqref{nablaMA}, \eqref{nablaVA} and \eqref{Immx} in \eqref{relation_G} gives
\begin{multline}\label{relation_G2}
\mathbb{E}\big[\|\nabla {m}_N^{n+1}\|_{2,x}^2\big] +\frac{1-\epsilon}{{\Delta t}}\mathbb{E}\big[\| {v}_N^n- {A}_N^n\|_{2,x}^2\big]+(2\theta-1-\theta\epsilon)\mathbb{E}\left[\|\nabla  {v}_N^n\|_{2,x}^2\right] \\
\leq \mathbb{E}\big[\|\nabla {m}_N^n\|_{2,x}^2\big]+\frac{{C\theta \Delta t}}{\epsilon}\|G\|_{2,2}^2 \big(\mathbb{E}\left[\|\nabla   {m}_N^n\|_{2,x}^2\right]+1\big) +\frac{{\Delta t}}{4\epsilon}\|G\|_{2,2}^4\,.
\end{multline}
Since $\theta\leq 1$, this proves the lemma.
\end{proof}

Let us fix $N\in\mathbb{N}^*$. Let $w_N^n=\frac{1}{\Delta t}({v}_N^n- {A}_N^n)$ for $n=0,...,N$.
In the following, we will also denote by ${v}_N$, $ w_N$, the piecewise constant processes (indexed by the time interval$[0,T]$), whose values on $[n{\Delta t},(n+1){\Delta t})$ are (respectively) $ {v}_N^n$, $  {w}_N^n$. The previous energy estimates can now be written in the form~:
\begin{equation}\label{estim_nabla_u_2_n}
\underset{t\in[0,T]}{\mathrm{esssup}}\hspace{0.2cm}\mathbb{E}\left[\|\nabla  m_N(t,\cdot)\|_{2,x}^2\right]\leq C\,,
\end{equation}
\begin{equation}\label{estim_w_2_n}
\mathbb{E}\left[\int_0^T\| {w_N} (t,\cdot)\|_{2,x}^2\dt\right]\leq C\,,
\end{equation}
\begin{equation}\label{estim_v_2_n}
\mathbb{E}\left[\int_0^T\| {v}_N (t,\cdot)\|_{2,x}^2 \dt\right]\leq C{\Delta t}\,,
\end{equation}
and
\begin{equation}\label{estim_nabla_v_2_n}
\mathbb{E}\left[\int_0^T\|\nabla  {v}_N (t,\cdot)\|_{2,x}^2 \dt\right]\leq C{\Delta t}.
\end{equation}
Note in addition that, since  $|m_N(t,x)|=1 $ for a.e.  $(\omega, t,x)\in \Omega\times [0,T]\times D$, one has
\begin{equation}\label{estim_u_2_n}
\mathbb{E}\left[\int_0^T\| m_N(t,\cdot)\|_{2,x}^2\dt\right]\leq C\,.
\end{equation}

\section{Tightness}
\label{sec:tightness}
The aim of this section is to show that the sequence $(m_N)_{N\in\mathbb{N}^*}$ is tight. Applying then the classical Prokhorov and Skorohod theorems (see for instance \cite{DAPRATOZABCZYCK2008}), we first get the relative compactness of the sequence of laws, and secondly we can assume the almost sure convergence to a certain limit $\bar m$, up to a change of probability space. In the sequel, we use the following notations :
for all $t\in[0,T]$, we set
 \begin{equation}\label{nota_XN}
  X_N(t):=\sum_{0\leq(n+1){\Delta t}\leq t} {m}_N^n\times(G{\Delta{W}_N^n)}=\sum_{0\leq(n+1){\Delta t}\leq t}{A}_N^n\in L^2(\Omega ;H^1(D)^3)\,,
\end{equation}
and
 \begin{equation}\label{nota_XNn}
  X_N^n:=X_N(n\Delta t)=\sum_{0\leq k\leq n-1} {m}_N^k\times(G{\Delta{W}_N^k)}\,.
\end{equation}

The process $t\mapsto X_N(t)$ is the martingale part of the semi-martingale $m_N$. It is a martingale with respect to a natural piecewise constant filtration, and corresponds to the noise induced fluctuations of the process $t\mapsto m_N(t)$. In order to get an almost sure convergence for the martingale part $X_N$, we consider the triplet $(m_N,X_N,GW)_{N\in\mathbb{N}^*}$, and show that it forms a tight sequence on a suitable space. This classical technique is used essentially to retrieve the noise term in this new probability space. This has the drawback that the new Wiener process depends on the integer $N\in\mathbb{N}^*$.
\begin{proposition}\label{pro_tension}
The sequence
$$(m_N, X_N,GW)_{N\in\mathbb{N}^*}$$ is tight in the space
$$L^2\big([0,T];L^2(D)^3\big)\times L^2\big([0,T];L^2(D)^3\big)\times\mathcal{C}\big(0,T;L^2(D)^3\big)\,.$$
\end{proposition}
The following result whose proof can be found e.g.\ in \cite{FLANDOLI_GATAREK1995}  will be needed for the proof of Proposition \ref{pro_tension}, which will be done later on.
\begin{lemma}\label{lem_flandoli_1}
 Let $\displaystyle B_0\subseteq B$ be two reflexive Banach spaces such that 
 $B_0$ is compactly embedded in $B$.
Let $\alpha>0$.
Then the embedding $$L^2(0,T;B_0)\cap H^\alpha(0,T;B)\hookrightarrow L^2(0,T;B)$$ is compact.
\end{lemma}
We shall apply this lemma with $B=L^2(D)^3$, and $B_0=H^1(D)^3$. Therefore, in order to deduce the tightness, we need uniform $H^\alpha(0,T;L^2(D)^3)$ estimates on $m_N$ and $X_N$ for some $\alpha>0$. These estimates are stated in the following proposition.
\begin{proposition}\label{pro_H_alpha}
For any $\alpha\in (0,\frac{1}{2})$,
there exists a constant $C=C(\|G\|_{2,2},T,$ $\alpha)$ such that
\begin{equation}\label{estimate1}
\mathbb{E}\left[\| m_N\|_{H^\alpha(0,T;L^2( D)^3)}^2\right]\leq C,
\end{equation}
\begin{equation}\label{estimate11}
\mathbb{E}\left[\| X_N\|_{H^\alpha(0,T;L^2( D)^3)}^2\right]\leq C.
\end{equation}
\end{proposition}
\begin{proof}[Proof of Proposition \ref{pro_H_alpha}.]
We have to evaluate the following quantities for $\alpha\in(0,\frac{1}{2})$~:
$$\iint_{[0,T]^2} \frac{\mathbb{E}\big[ \| m_N(t)- m_N(s)\|_{2,x}^2\big]}{|t-s|^{1+2\alpha}}\dt\ds\,.$$
and
$$\iint_{[0,T]^2} \frac{\mathbb{E}\big[ \| X_N(t)- X_N(s)\|_{2,x}^2\big]}{|t-s|^{1+2\alpha}}\dt\ds\,.$$
Notice that these integrals measure the regularity in time of the two processes $t\mapsto m_N(t)$ and $t\mapsto X_N(t)$. Since $X_N$ is expected to be the martingale part in the canonical decomposition of the semi-martingale $m_N$, one expects $m_N$ to be at least as regular as $X_N$. In the sequel, we take 
$t,s\in[0,T]$, and assume without loss of generality that $t>s$.
Thus, we first evaluate the following quantity~:
\begin{equation}\label{increments_X_N}
\mathbb{E}\Big[\| X_N(t)- X_N(s)\|_{2,x}^2\Big]=\mathbb{E}\Big[\Big\|\sum_{s<(n+1){\Delta t}\leq t}{A}_N^n\Big\|_{2,x}^2\Big]\,.
\end{equation}
Observe that for $n\neq m$, the random variables
$G\Delta W_N^n$, $G\Delta W_N^m$ are independent, with zero mean. Using also the fact that ${m}_N^n$ is independent of $G\Delta W_N^k$ for $1\leq n\leq k\leq N$, 
Fubini's theorem, and the identity $a\cdot(b\times c)=a\times b\cdot c$, for  $a,b,c \in\mathbb{R}^3$, one has for $m>n$,
\begin{flalign}
\nonumber\mathbb{E}\big[\big({A}_N^n&,{A}_N^m\big)_{2,x}\big]&\\
&=\int_D\mathbb{E}\left[\big({m}_N^n(x)\times G\Delta W_N^n(x)\big)\cdot\big({m}_N^m(x)\times G\Delta W_N^m(x)\big)\right]\dx&\\
\nonumber&=\int_D\mathbb{E}\left[\big(\big({m}_N^n(x)\times G\Delta W_N^n(x)\big)\times{m}_N^m(x) \big)\right]\cdot \mathbb{E}\left[G\Delta W_N^m(x)\right]\dx&\\
\label{estim_ind_increments}&=0\,.&
\end{flalign}
Developing the sum \eqref{increments_X_N} and using \eqref{estim_A_2}, one has
\begin{eqnarray*}
\mathbb{E}\Big[\| X_N(t)- X_N(s)\|_{2,x}^2\Big]&=&\mathbb{E}\Big[\sum_{s<(n+1){\Delta t}\leq t}\|{A}_N^n\|_{2,x}^2\Big]\\
&&+2\mathbb{E}\Bigg[\sum_{\substack {s<(n+1){\Delta t}\leq t\\s<(m+1){\Delta t}\leq t\\n<m}}\big({A}_N^n,{A}_N^m\big)_{2,x}\Bigg]\\
&=&\mathbb{E}\Big[\sum_{s<(n+1){\Delta t}\leq t}\|{A}_N^n\|_{2,x}^2\Big]\\
&\leq& C\|G\|_{2,0}^2\Big(\sum_{s<(n+1){\Delta t}\leq t}{\Delta t}\Big)\,.
\end{eqnarray*}
We observe that the number of terms in the sum above is bounded by $\dfrac{t-s}{\Delta t}+1$, and deduce that for all $0\leq s\leq t\leq T$,
\begin{eqnarray}\label{increments_X_N_2}
 \mathbb{E}\Big[\| X_N(t)- X_N(s)\|_{2,x}^2\Big]\leq C\|G\|_{2,0}^2(|t-s|+\Delta t)\,.
\end{eqnarray}
Now, remark that \eqref{increments_X_N_2} implies the uniform estimate of the $H^\alpha$ norm. Indeed,
 since $X_N$ is a piecewise constant function, the integrand $\mathbb{E}\big[\| X_N(t,\cdot)- X_N(s,\cdot)\|_{2,x}^2\big]$ vanishes for $(t,s)\in[n\Delta t,(n+1)\Delta t)^2$, for $0\leq n\leq N-1$. Using moreover \eqref{increments_X_N_2}, there exists a constant $C=C(\|G\|_{2,0},T)$ such that
\begin{flalign*}
\nonumber&\iint_{[0,T]^2}\frac{\mathbb{E} [\|X_N(t)-X_N(s)\|_{2,x}^2]}{|t-s|^{1+2\alpha}}\dt\ds&\\
&\hspace{0.2cm}\nonumber\leq C\iint_{[0,T]^2}\frac{\dt\ds}{|t-s|^{2\alpha}}
+C\Delta t \sum_{\substack{0\leq m,n\leq N-1\\|n-m|\geq 1}}\int_{n\Delta t}^{(n+1)\Delta t}\int_{m\Delta t}^{(m+1)\Delta t}\frac{\dt\ds}{|t-s|^{1+2\alpha}}&\\
&\hspace{0.2cm} =A+B \,.&
\end{flalign*}
Since 
\begin{equation*}
\bigcup\limits_{\substack{0\leq n,m\leq N-1\\|n-m|\geq2}}[n\Delta t,(n+1)\Delta t[\times[m\Delta t,(m+1)\Delta t[\subseteq\{(t,s)\in[0,T]^2,|t-s|>\Delta t\}\,,
\end{equation*}
we remark that
\begin{multline*}
 B \leq \sum_{\substack{0\leq m,n\leq N-1\\|n-m|= 1}}\int_{n\Delta t}^{(n+1)\Delta t}\int_{m\Delta t}^{(m+1)\Delta t}\frac{ C\Delta t\dt\ds}{|t-s|^{1+2\alpha}}+\iint\limits_{\substack{[0,T]^2\\|t-s|>\Delta t}}\frac{C\Delta t}{|t-s|^{1+2\alpha}}\dt\ds\,.
\end{multline*}
Finally, we get 
\begin{multline}
\label{integ1}
\iint_{[0,T]^2}\frac{\mathbb{E} [\|X_N(t)-X_N(s)\|_{2,x}^2]}{|t-s|^{1+2\alpha}}\dt\ds\\
\leq 2C \iint_{[0,T]^2}\frac{\dt\ds}{|t-s|^{2\alpha}}+2C\Delta t\sum_{n=0}^{N-2}\int_{n\Delta t}^{(n+1)\Delta t}\int_{(n+1)\Delta t}^{(n+2)\Delta t}\frac{\dt\ds}{|t-s|^{1+2\alpha}}\,.
\end{multline}

The first term of the right hand side of \eqref{integ1} is bounded because $\alpha\in(0,\frac{1}{2})$.
Then, it is easy to show that
\begin{equation}\label{ineg2}
\sum_{n=0}^{N-2}\int_{n\Delta t}^{(n+1)\Delta t}\int_{(n+1)\Delta t}^{(n+2)\Delta t}\frac{\dt\ds}{|t-s|^{1+2\alpha}}=O(\Delta t^{-2\alpha})\,,
\end{equation}
and \eqref{estimate11} is proved.

We now turn to \eqref{estimate1}. Note that since we have already estimated $X_N$, it remains only to consider
$m_N-X_N$.
Using the definition of $m_N$ and $X_N$ together with \eqref{pv_renorm}, we write~:
\begin{flalign*}
 (m_N(t)-X_N(t))-&(m_N(s)-X_N(s))&\\
 =&\sum_{s<(n+1){\Delta t}\leq t}\Big({m}_N^{n+1}- {m}_N^n-{A}_N^n\Big)&\\
=&\sum_{s<(n+1){\Delta t}\leq t}\Big(\frac{{m}_N^{n}+{A}_N^n}{|{m}_N^{n}+{A}_N^n|}- {m}_N^n-{A}_N^n\Big)&\\
&+\sum_{s<(n+1){\Delta t}\leq t}\left(\frac{{m}_N^{n}+{v}_N^n}{|{m}_N^{n}+{v}_N^n|}-\frac{{m}_N^{n}+{A}_N^n}{|{m}_N^{n}+{A}_N^n|}\right)\,.&
\end{flalign*}
Then, taking the $L^2(D)^3$ norm, and the expectation, we get
\begin{flalign}
\mathbb{E}\big[\| m_N(t)-X_N(t)-& (m_N(s)-X_N(s))\|_{2,x}^2\big] &\\
\leq&  2\mathbb{E}\big[\big\|\sum_{s<(n+1){\Delta t}\leq t} \frac{{m}_N^{n}+{A}_N^n}{|{m}_N^{n}+{A}_N^n|} -{m}_N^n-{A}_N^n\big\|_{2,x}^2\big]&\nonumber\\
&+2\mathbb{E}\big[\big\|\sum_{s<(n+1){\Delta t}\leq t}\frac{{m}_N^{n}+{v}_N^n}{|{m}_N^{n}+{v}_N^n|}-\frac{{m}_N^{n}+{A}_N^n}{|{m}_N^{n}+{A}_N^n|}\big\|_{2,x}^2\big]\,.&\label{increments_m_N}
\end{flalign}

For the first term in the right hand side of \eqref{increments_m_N}, observe that for any ${m},{V}\in\mathbb{R}^3$, s.t.\ ${V}\perp{m}$ and $|{m}|=1$, one has~:
\begin{equation}\label{ineg_accroissements}
 \left| \frac{{m}+{V}}{|{m}+{V}|}- {m}-{V}\right|\leq \sqrt{1+|V|^2}-1\leq \frac{1}{2}|{V}|^2\,.
 \end{equation} 

Using Cauchy-Schwarz inequality, and \eqref{ineg_accroissements} on each term of the sum (remember that ${A}_N^n\perp{m}_N^n$, see \eqref{nota_AN}), one has~:
\begin{flalign*}
 E\Big[\Big\|\sum_{s<(n+1){\Delta t}\leq t} &\frac{{m}_N^{n}+{A}_N^n}{|{m}_N^{n}+{A}_N^n|} -{m}_N^n-{A}_N^n\Big\|_{2,x}^2\Big]&\\
\leq&\Big(\frac{t-s}{\Delta t}+1\Big)\sum_{s<(n+1){\Delta t}\leq t}\mathbb{E}\Big[\Big\|\frac{{m}_N^{n}+{A}_N^n}{|{m}_N^{n}+{A}_N^n|} -{m}_N^n-{A}_N^n\Big\|_{2,x}^2\Big]&\\
 \leq&\frac{1}{4}\Big(\frac{t-s}{\Delta t}+1\Big)\sum_{s<(n+1){\Delta t}\leq t}\mathbb{E}\Big[\Big\|{A}_N^n\Big\|_{4,x}^4\Big]\,.&
\end{flalign*}
Then we have by \eqref{estim_A_4}~:
\begin{align}\label{maj1}
E\Big[\Big\|\sum_{s<(n+1){\Delta t}\leq t} \frac{{m}_N^{n}+{A}_N^n}{|{m}_N^{n}+{A}_N^n|} -{m}_N^n-{A}_N^n\Big\|_{2,x}^2\Big]
&\leq C(\|G\|_{2,1},T)(|t-s|+\Delta t)^2\,.
\end{align}

Similarly, for the second term in \eqref{maj1}, we use the fact that the map $\displaystyle x\mapsto\frac{x}{|x|}$, is $1$-Lipschitz for $|x|\geq1$, together with Cauchy-Schwarz inequality and \eqref{estim_w}. Then
\begin{eqnarray}
& &\nonumber \mathbb{E}\Big[\Big\|\sum_{s<(n+1){\Delta t}\leq t}\frac{{m}_N^{n}+{v}_N^n}{|{m}_N^{n}+{v}_N^n|}-\frac{{m}_N^{n}+{A}_N^n}{|{m}_N^{n}+{A}_N^n|}\Big\|_{2,x}^2\Big]\\
\label{maj2}&& \hspace{1cm}\leq\Big(\frac{t-s}{\Delta t}+1\Big)\mathbb{E}\Big[\sum_{s<(n+1){\Delta t}\leq t}\Big\|{v}_N^n-{A}_N^n\Big\|_{2,x}^2\Big]\\
\nonumber&& \hspace{1cm}\leq C(\|G\|_{2,2},T)(|t-s|+\Delta t)\,.
\end{eqnarray}

Using \eqref{maj1} and \eqref{maj2} in \eqref{increments_m_N}, together with \eqref{estimate11}, we conclude that there exists a constant $C=C(\|G\|_{2,2},T)$, independent of $N\in\mathbb{N}^*$, such that
\begin{equation*}
 \mathbb{E}\left[\| m_N(t,\cdot)- m_N(s,\cdot)\|_{2,x}^2\right]\leq C(|t-s|+\Delta t)\,.
\end{equation*}
We have proved the same inequality as for the process $X_N$ (see \eqref{increments_X_N_2}), thus the conclusion follows in the same way as before.
\end{proof}
We now turn to the proof of Proposition \ref{pro_tension}.
\begin{proof}[Proof of Proposition \ref{pro_tension}]
Remark that thanks to \eqref{estim_u_2_n} and \eqref{estim_nabla_u_2_n},\linebreak $(m_N)_{N\in\mathbb{N}^*}$ is bounded in $L^2(\Omega\times[0,T];H^1(D)^3)$. Let us prove the same for the process $X_N$.
One has~:
\begin{eqnarray*}
 \mathbb{E}\Big[\| X_N\|_{L^2(0,T;L^2(D)^3)}^2\Big]&=&\mathbb{E}\Big[\sum_{n=0}^{N-1}\|{X}_N^n\|_{2,x}^2{\Delta t}\Big]\\
&=&\sum_{n=0}^{N-1}{\Delta t}\,\mathbb{E}\Big[\Big\|\sum_{k\leq n }{A}_N^k\Big\|_{2,x}^2\Big]\\
&=&\sum_{n=0}^{N-1}{\Delta t}\Big(\mathbb{E}\Big[\sum_{k\leq n}\Big\|{A}_N^k\Big\|_{2,x}^2\Big]\\
&&+2\mathbb{E}\Big[\sum_{0\leq k<l\leq n}({A}_N^k,{A}_N^l)_{2,x}\Big]\Big)\,.
\end{eqnarray*}
As before, the second term vanishes (see \eqref{estim_ind_increments}), while the first term is bounded by $C(\|G\|_{2,0},T)$ thanks to \eqref{estim_A_2}.

Similarly, for $k\neq l$, we have
$\mathbb{E}[(\nabla {A}_N^k,\nabla {A}_N^l)_{2,x}]=0\,.$
Moreover, using \eqref{estim_nabla_A}, we get
\begin{eqnarray*}
 \mathbb{E}[\|\nabla X_N\|_{2,x}^2]&=&\sum_{n=0}^{N-1}{\Delta t}\sum_{k\leq n}\mathbb{E}\Big[\|\nabla {A}_N^k\|_{2,x}^2\Big]\\
&\leq& C\|G\|_{2,2}^2\sum_{n=0}^{N-1}{\Delta t}\sum_{k\leq n}{\Delta t}\\
&\leq& C'(\|G\|_{2,2},T)\,.
\end{eqnarray*}
Therefore, there exists a constant $C=C(\|G\|_{2,2},T)>0$ such that 
 $$\mathbb{E}\Big[\| X_N\|_{ L^2(0,T;H^1(D)^3)}^2\Big]\leq C\,.$$

The tightness of the sequence $(m_N,X_N,W)$ is now obtained in a classical way.
Let $R>0$, and fix $\alpha\in(0,\frac{1}{2})$.
We consider the product space
$$E:=L^2\big(0,T; L^2(D)^3\big)\times L^2\big(0,T; L^2(D)^3)\big)\times \mathcal{C}\big([0,T];L^2(D)^3\big)\,,$$ endowed with its classical product norm.
Thanks to lemma \ref{lem_flandoli_1}, and a standard Ascoli compactness theorem, the space 
\begin{multline*}
F:=L^2\big(0,T;H^1(D)^3\big)\times L^2\big(0,T;H^1(D)^3\big)\times\mathcal{C}\big([0,T];H^1(D)^3\big)\\
 \bigcap H^\alpha\big(0,T;L^2(D)^3\big)\times H^\alpha\big(0,T;L^2(D)^3\big)\times\mathcal{C}^\alpha\big([0,T];L^2(D)^3\big)
\end{multline*}
is compactly embedded in $E$.
Using Markov inequality, one has
\begin{flalign}\nonumber
\mathbb{P}&\Big(( m_N, X_N, GW)\notin B_F(0,R)\Big)&\\
&\label{Markov}\leq \frac{1}{R^2}\Big(\mathbb{E}\left[\|m_N\|_{ L^2(0,T;H^1)}^2\right]+\mathbb{E}\left[\|m_N\|_{H^\alpha([0,T];L^2)}^2\right]& \nonumber\\
&+\mathbb{E}\left[\|X_N\|_{ L^2(0,T;H^1)}^2\right]+\mathbb{E}\left[\|X_N\|_{ H^\alpha([0,T];L^2)}^2\right]+\mathbb{E}\left[\|GW\|_{\mathcal{C}^\alpha([0,T];H^1)}^2\right]\Big)\,.&
\end{flalign}
Then, using the bounds \eqref{estimate1}, and \eqref{estimate11}, \eqref{G_hypothesis}, the classical properties of a $GG^*$-Wiener process and also \eqref{estim_u}, the right hand side of \eqref{Markov} tends to $0$ as $R\to\infty$ uniformly in $N\in\mathbb{N}^*$.
Since the sets $B_F(0,R)$ are precompacts in $E$, the sequence $(m_N,X_N,GW)_{N\in\mathbb{N}^*}$ is tight in $E$, and the proposition is proved.
\end{proof}
A simple application of Prokhorov and Skorohod theorem leads to the following corollary~:
\begin{corollary}\label{cor_skorohod}
There exists a new probability space 
$(\bar\Omega,\mathcal{\bar F},\mathbb{\bar P})\,,$
a sequence of random variables on this space
$(\bar m_N, \bar X_N, G\bar W_N)_{N\in\mathbb{N}^*}$
 taking its values in the space $L^2(0,T;L^2(D)^3)\times L^2(0,T;L^2(D)^3)\times \mathcal{C}(0,T;L^2(D)^3)$, with the same laws, for each $N\in\mathbb{N}^*$, as $( m_N, X_N, GW)$,
and a triplet $({\bar m},{\bar X},G\bar W)$ of r.v.\ in $L^2(0,T;L^2(D)^3)\times L^2(0,T;L^2(D)^3)\times \mathcal{C}(0,T;L^2(D)^3)\,,$
so that up to a subsequence,
\begin{equation*}
  {\bar m}_N \underset{N\to\infty}{\longrightarrow} \bar m\quad\text{a.s.\ in }L^2\big([0,T];L^2(D)^3\big)\,,
\end{equation*}
\begin{equation*}
  {\bar X}_N \underset{N\to\infty}{\longrightarrow} \bar X\quad\text{a.s.\ in }L^2\big([0,T];L^2(D)^3\big)\,,
\end{equation*}
\begin{equation*}
 G{\bar W}_N\underset{N\to\infty}{\longrightarrow}G{\bar W}\quad\text{a.s.\ in }\mathcal{C}\big(0,T;L^2(D)^3\big)\,.
\end{equation*} 
\end{corollary}

Since $m_N$, $X_N$ are piecewise constant processes, the same is also true for their counterparts in the new probability space $\bar\Omega$. We define the following discrete parameter processes, for $0\le n \le N$ :
$$  {\bar m}_N^n:= \bar m_{N}(n{\Delta t}) \in L^2(D)^3\,,$$
$$ {\bar X}_N^n:= \bar X_{N}(n{\Delta t}) \in L^2(D)^3\,,$$
and also
$$G\Delta \bar W_N^n:=G\bar W_N((n+1)\Delta t)-G\bar W_N(n\Delta t)\,,$$
$${\bar A}_N^n:={\bar m}_N^n\times(G\Delta\bar W_N^n)\,,$$
and
${\bar v}_N^n$ as the unique solution of \eqref{pv} associated to the data $({\bar m}_N^n\,,G\Delta \bar W_N^n)$, i.e.\ for all $0\leq n\leq N$, and all $\varphi\in\mathbb{\bar W}_{N,n}$,
\begin{flalign}
 \label{pv_bar}
\Big(\bar v_N^n-  {\bar m}_N^n\times  &{\bar v}_N^n,\varphi\Big)_{2,x} +2\theta{\Delta t}\Big(\nabla  {\bar v}_N^n ,\nabla\varphi\Big)_{2,x} &\nonumber \\
=&-2{\Delta t}\Big(\nabla {\bar m}_N^n,\nabla\varphi\Big)_{2,x}+ \Big((\mathrm{Id}- {\bar m}_N^n\times)\big(  {\bar m}_N^n\times G{\Delta{\bar W}_N^n}\big),\varphi\Big)_{2,x}&\nonumber \\
 & +\frac{{\Delta t}}{2}\sum_{i\in\mathbb{N}}\Big((\mathrm{Id}-{\bar m}_N^n\times)\big(({\bar m}_N^n\times G_i)\times G_i\big),\varphi\Big)_{2,x}\,,&
\end{flalign}
where
$$\mathbb{\bar W}_{N,n}(\omega):=\Big\{ \psi\in H^1(D)^3, \;   \forall x\in D, \;  \psi(x)\perp{\bar m}_N^n(\omega,x) \Big\}.$$
These random variables have the same laws as their counterparts in $\Omega$ that is (respectively) ${m}_N^n$, ${X}_N^n$, $G\Delta W_N^n$ and ${A}_N^n:={m}_N^n\times(G\Delta W_N^n)$. We already noticed that ${ v}_N^n$ depends continuously on the couple $({m}_N^n\,,G\Delta \bar W_N^n)$ through \eqref{pv}, and thus the law of ${\bar v}_N^n$ is the same as the law of ${v}_N^n$. It also follows that we have the identity
 \begin{equation}
 \label{renormbar}
 {\bar m}_N^{n+1}=\frac{ {\bar m}_N^n+{\bar v}_N^n}{|{\bar m}_N^n+{\bar v}_N^n|}\quad \text{a.s.}                                              \end{equation} 
We still need to define the following processes on $\bar\Omega$~:
$\forall t\in[0,T]$,
\begin{equation}\label{def_bar_v}
  \bar {v}_N(t):={\bar v}_N^n\quad \text{if  } t\in[n{\Delta t},(n+1){\Delta t}),
\end{equation} 
and 
\begin{equation}\label{def_bar_w}
 {\bar w}_N(t):=\dfrac{{\bar v}_N^n-{\bar A}_N^n}{{\Delta t}}\quad \text{if } t\in[n{\Delta t},(n+1){\Delta t})\,.
\end{equation}
\begin{remark}\label{rema_conv_nabla}
By \eqref{estim_nabla_u_2_n} and a classical compactness argument, we may assume that up a subsequence the following convergence holds
\begin{equation}\label{weak_convergence_nabla_m}
 \nabla\bar m_N\underset{N\to\infty}{\rightharpoonup}\nabla\bar m\quad\text{weakly in }L^2(\bar\Omega\times[0,T]\times D)^{3\times3}\,.
\end{equation}
\end{remark}

\section{Convergence of the martingale part}
\label{sec:convergence_of_the_martingale_part}

In section 4 we proved that the process $\bar X_N$ converges almost surely in $L^2([0,T]$ $\times D)^3$ to $\bar X$. Here we show that $\bar X$ defines a square integrable continuous martingale with values in $L^2(D)^3$.  We define the filtration $(\mathcal{\bar F}_t)_{t\in[0,T]}$ as
\begin{equation}\label{continuous_filtration}
 \mathcal{\bar F}_t=\sigma\left\{G\bar W(s)\,,s\leq t\right\}\,.
\end{equation}
\begin{proposition}\label{pro_martingale_barX}
The process $\displaystyle t\in[0,T]\mapsto {\bar X}(t,\omega)\in L^2(D)^3$ is a square integrable continuous martingale with respect to the filtration $(\mathcal{\bar F}_t)$, with quadratic variation defined for all $a,b\in L^2(D)^3$ by:
\begin{equation*}
\Big(\ll \bar X\gg_ta\,,\,b\Big)_{2,x}=\int_0^t\Big({\bar m}\times(Ga)\,,\,{\bar m}\times(Gb)\Big)_{2,x}\ds \,. 
\end{equation*} 
\end{proposition}
The proof needs an additional martingale-type uniform estimate on $X_N$.
\begin{proposition}\label{pro_moments_X}
For all $q\in\mathbb{N}$, there exists a constant $C=C(\|G\|_{2,0},T,q)>0$ independent of $N\in\mathbb{N}^*$, such that
\begin{equation*}
\mathbb{E}\left[ \max_{n\in\{0,\dots,N\}} \|{\bar X}_N^n\|_{2,x}^{2q}\right]\leq C\,.
\end{equation*}  
\end{proposition}
To prove proposition \ref{pro_moments_X}, we 
state a  discrete version of the Burkholder-Davis-Gundy inequality with values in a Hilbert space. The following result is a particular case of Proposition $2$ of \cite{ASSOUAD1975}, and we therefore omit the proof.
\begin{lemma}\label{lem_moments}
For a given discrete parameter martingale $(M^n)_{0\leq n\leq N}$ with values in a Hilbert space $H$, for any $q\in\mathbb{N}^*$, there exist $C=C(q)>0$ such that the following inequality holds~:
\begin{equation*}
\mathbb{E}\Big[\max_{0\leq n\leq N}\|M^n\|^{2q}_H\Big]\leq C\mathbb{E}\Big[\Big(\sum_{n=0}^{N-1}\|M^{n+1}-M^n\|_H^2\Big)^q\Big]\,.
 \end{equation*}
\end{lemma}
\begin{remark}
Since for all $N\in\mathbb{N}^*$, the laws of $X_N$ and $\bar X_N$ are equal, note that for all $t\in[0,T]$, and almost surely,
\begin{equation}\label{expression_of_bar_XN}  {\bar X}_N(t)=\sum_{0\leq(n+1){\Delta t}\leq t}{\bar m}_N^n\times{G}\Delta \bar W_N^n\,.
\end{equation}
It is easily seen, using \eqref{pv_bar} and \eqref{renormbar} that $(\bar m_N^n)_{0\leq n \leq N}$ is adapted to
\begin{equation}\label{discrete_filtration_bar}
 \mathbb{\bar F}_{N}^n=\sigma\left\{G\bar W_N(k\Delta t) ;  k\in\mathbb{N}^*,k\leq n\right\}\,,
\end{equation}
and the process $\bar X_N^n$ defines a martingale with respect to this filtration. In particular, we have the following identity~: for all $0\leq n\leq n'\leq N$, and any bounded continuous function $\phi$ on $(L^2(D)^3)^n$,
\begin{equation}\label{martingale_XN_1}
 \mathbb{E}\Big[\big(\bar X_N^{n'}-\bar X_N^{n}\big)\phi(G\bar W_N(\Delta t),\dots,G\bar W_N(n\Delta t))\Big]=0\,,
\end{equation}
The reader may also check that for any $n,n',\phi$ as above, and for all $a,b\in L^2(D)^3$,
\begin{multline}\label{martingale_XN_2}
\mathbb{E}\Big[\Big((\bar X_N^{n'},a)_{2,x}(\bar X_N^{n'},b)_{2,x}-(\bar X_N^n,a)_{2,x}(\bar X_N^n,b)_{2,x}\\
-\sum_{n\leq k\leq n'-1}{\Delta t}\big( {\bar m}_N^k\times(Ga)\,,\, {\bar m}_N^k\times(Gb)\big)_{2,x} \Big)\\ \phi(G\bar W_N(\Delta t),\dots,G\bar W_N(n\Delta t))\Big]=0.
\end{multline}
Equation \eqref{martingale_XN_2} gives us the quadratic variation of $(\bar X_N^n)_{0\leq n\leq N}$.
\end{remark}
\begin{proof}[Proof of Proposition \ref{pro_moments_X}]
 Assume that $N\in\mathbb{N}^*$ is given.
We apply Lemma \ref{lem_moments} to the discrete parameter martingale $\big({\bar X}_N^n\big)_{0\leq n \leq N}$, which takes values in the Hilbert space $H=L^2(D)^3$.
Thanks to \eqref{expression_of_bar_XN}, and H\"older's inequality, one has
\begin{align*}
 \mathbb{E}\Big[\Big(\sum_{n=0}^{N-1}\|\bar X_N^{n+1}-{\bar X}_N^n\|_{2,x}^2\Big)^q\Big]&=\mathbb{E}\Big[\Big(\sum_{n=0}^{N-1}\big\|{\bar m}_N^n\times(G\Delta\bar W_N^n)\big\|_{2,x}^{2}\Big)^q\Big]\\
&\leq  N^{q-1}\sum_{n=0}^{N-1}\mathbb{E}\Big[\|{\bar m}_N^n\times(G\Delta\bar W_N^n)\|_{2,x}^{2q}\Big]\,.
\end{align*}
It is known (see for instance \cite{DAPRATOZABCZYCK2008}, corollary 2.17 ) that since $G\Delta\bar W_N^n$ is a gaussian random variable with covariance $\Delta tGG^*$,  there exists a constant $C(\|G\|_{2,0},q)>0$ (independent of $n$ and $N$) such that :
\begin{equation}\label{moment_gauss_q}
 \mathbb{E}\Big[\big\|G\Delta\bar W_N^n\big\|_{2,x}^{2q}\Big]\leq C(\|G\|_{2,0},q)\Delta t^q\,.
\end{equation}
Thus one has, recalling that $|\bar m_N^n|=1 \text{ a.e.}$,
\begin{align*}
\mathbb{E}\Big[\Big(\sum_{n=0}^{N-1}\|\bar X_N^{n+1}-{\bar X}_N^n\|_{2,x}^2\Big)^q\Big] &\leq N\big(C(\|G\|_{2,0},q){\Delta t}^q\big)N^{q-1}\leq C'(\|G\|_{2,0},T,q)\,.
\end{align*}
This proves proposition \ref{pro_moments_X}.
\end{proof}
We now turn to the proof of Proposition \ref{pro_martingale_barX}.
\begin{proof}[Proof of Proposition \ref{pro_martingale_barX}.]
\textit{$\bar X$ is a martingale :}\\
We use equalities \eqref{martingale_XN_1} and \eqref{martingale_XN_2}. 
We have to show that for any bounded continuous function $\phi$ defined on the space $(L^2(D)^3)^K$, any $a,b\in L^2(D)^3$ the following relations hold for almost all $0\leq s\leq t\leq T$, all $K\in\mathbb{N}^*$, and $t_1\leq \dots t_K<s$~:
\begin{equation}\label{martingale_bar_X}
\mathbb{E}\left[( {\bar X}(t)- {\bar X}(s))\phi\big(G\bar W(t_1),\dots,G\bar W(t_K)\big)\right]=0\,.
\end{equation}
and
\begin{multline}\label{martingale_bar_X_var_quad}
\mathbb{E}\Big[\Big(({\bar X}(t),a)_{2,x}({\bar X}(t),b)_{2,x}-({\bar X}(s),a)_{2,x}({\bar X}(s),b)_{2,x}\\
-\int_s^t \big({\bar m}(\sigma)\times(Ga)\,,\, {\bar m}(\sigma)\times(Gb)\big)_{2,x} \mathrm{d}\sigma\Big)\phi\big(G\bar W(t_1),\dots,G\bar W(t_K)\big)\Big]=0\,.
\end{multline}
First, observe that as a consequence of Proposition \ref{pro_moments_X}, and Egorov's Theorem, 
\begin{equation}\label{conv_egoroff}
 {\bar X}_N\underset{N\to\infty}{\longrightarrow}\bar X \text{ in } L^2(\bar\Omega\times[0,T]\times D)^3\,.
\end{equation}

Hence, up to a subsequence, one has for almost all $t,s\in[0,T]$,
\begin{equation*}
{\bar X}_N(t)- {\bar X}_N(s)\underset{N\to\infty}{\longrightarrow}{\bar X}(t)-{\bar X}(s),\quad \text{in }L^2(\bar\Omega\times D)^3\,.
\end{equation*}
For all $0\leq k\leq K$, if $\left[\frac{t_k}{\Delta t}\right]$ denotes the floor of $\frac{t_k}{\Delta t}$, then $\left[\frac{t_k}{\Delta t}\right]\Delta t$ tends to $t_k$ as $N\to\infty$. Taking into account the almost sure continuity of the limit process $G\bar W$, and the fact that the process $G\bar W_N$ converges almost surely to ${G\bar W}$ in $\mathcal{C}([0,T]; L^2(D)^3)$ as $N$ tends to $\infty$, one has 
\begin{multline*}\textstyle \left(G\bar W_N\left(\left[\frac{t_1}{\Delta t}\right]\Delta t\right),\cdots,G\bar W_N\left(\left[\frac{t_K}{\Delta t}\right]\Delta t\right)\right)\\ \underset{N\to\infty}{\longrightarrow}(G\bar W(t_1),\cdots,G\bar W(t_K))\text{ in }(L^2(D)^3)^K\,.
\end{multline*}
The application $\phi$ being continuous, we conclude that
\begin{multline*}
\textstyle
 \mathbb{E}\left[(\bar X_N(t)-\bar X_N(s))\phi\left(G\bar W_N\left(\left[\frac{t_1}{\Delta t}\right]\Delta t\right),\cdots,G\bar W_N\left(\left[\frac{t_K}{\Delta t}\right]\Delta t\right)\right)\right]\\
 \underset{N\to\infty}{\longrightarrow}\mathbb{E}\left[(\bar X_N(t)-\bar X_N(s))\phi(G\bar W(t_1),\cdots,G\bar W(t_K))\right]\,.
\end{multline*}
On the other hand, 
by \eqref{martingale_XN_1}
\begin{equation*}
\textstyle
 \mathbb{E}\left[(\bar X_N(t)-\bar X_N(s))\phi\left(G\bar W_N\left(\left[\frac{t_1}{\Delta t}\right]\Delta t\right),\cdots,G\bar W_N\left(\left[\frac{t_K}{\Delta t}\right]\Delta t\right)\right)\right]
 =0\,,
\end{equation*}
and \eqref{martingale_bar_X} is proved.

If $a,b\in L^2([0,T]\times D)^3$, then \eqref{martingale_XN_2} implies that~:
\begin{multline*}
\textstyle
\mathbb{E}\Big[\Big(( {\bar X}_N(t),a)_{2,x}( {\bar X}_N(t),b)_{2,x}-( {\bar X}_N(s),a)_{2,x}( {\bar X}_N(s),b)_{2,x}\\
-\sum_{s<(n+1){\Delta t}\leq t}\textstyle \Delta t\big( {\bar m}_N^n\times(Ga)\,,\, {\bar m}_N^n\times(Gb)\big)_{2,x}\Big) \\ \phi(G\bar W_N([\frac{t_1}{\Delta t}]\Delta t),\cdots,G\bar W_N([\frac{t_K}{\Delta t}]\Delta t))\Big]=0\,.
\end{multline*} 
Moreover,
\begin{equation*}
 ( {\bar X}_N(t),a)_{2,x}( {\bar X}_N(t),b)_{2,x}-( {\bar X}_N(s),a)_{2,x}( {\bar X}_N(s),b)_{2,x}
 \end{equation*}
 tends to
\begin{equation*}
({\bar X}(t),a)_{2,x}({\bar X}(t),b)_{2,x}
-({\bar X}(s),a)_{2,x}({\bar X}(s),b)_{2,x}
\end{equation*}
in $L^1(\bar\Omega)$, while
\begin{equation*}
 \sum_{s<(n+1){\Delta t}\leq t}\Delta t\big( {\bar m}_N^n\times(Ga)\,,\, {\bar m}_N^n\times(Gb)\big)_{2,x}
 \end{equation*}
 converges to
\begin{equation*}
\int_s^t \big( {\bar m}(\sigma)\times{G a}\,,\,{\bar m}(\sigma)\times{G b}\big)_{2,x}\mathrm{d}\sigma,
\end{equation*}
in $L^1(\bar\Omega)$.
This proves \eqref{martingale_bar_X_var_quad}.
It remains to prove that $\bar X$ has continuous trajectories.\bigskip\\
\textit{Proof of the continuity.}\\
We prove that the limit $\bar X$ satisfies the assumptions of Kolmogorov's test (see e.g. \cite{DAPRATOZABCZYCK2008}, theorem 3.3). More precisely, we show that for any $q\in\mathbb{N}^*$, there exists $C_q>0$, such that for almost every $(t,s)\in[0,T]^2$,
\begin{equation}\label{kolmogorov_test}
 \mathbb{E}\left[\|{\bar X}(t)-{\bar X}(s)\|_{2,x}^{2q}\right]\leq C_q|t-s|^q\,.
\end{equation} 
Let $T\geq t>s\geq 0,\quad $ and $n,n'\in\mathbb{N}$, the unique integers such that 
$t\in[n'{\Delta t},(n'+1){\Delta t})$ and $s\in[n{\Delta t},(n+1){\Delta t}[$. One has $|t-n'{\Delta t}|\leq\Delta t$ and $|s-n{\Delta t}|\leq{\Delta t}$.
We consider the discrete parameter martingale which starts at $n\Delta t$, and whose increments are the same as $\big({\bar X}_N^k\big)_{k\in\{0,\dots,N\}}$. More precisely, let $\displaystyle (M_N^l)_{0\leq l\leq n'-n}$ be the discrete parameter process defined by
$$
  \bar M_N^l={\bar X}_N^{n+l}-{\bar X}_N^n 
 = \sum_{k=n+1}^{n+l}\bar A_N^k, \quad \text{for}\quad 0\leq l\leq  n'-n \,.
$$
The process $\displaystyle (\bar M_N^l)_{0\leq l\leq n'-n}$ defines a martingale for the discrete filtration  \linebreak$(\mathbb{\bar F}_{(l+n){\Delta t}})_{0\leq l\leq n'-n}\,,$
(see \eqref{discrete_filtration}).
Using similar arguments as for the proof of Proposition \ref{pro_moments_X}, and in particular Lemma \ref{lem_moments},
\begin{eqnarray*}
\mathbb{E}\left[\| {\bar X}_N(t)- {\bar X}_N(s)\|_{2,x}^{2q}\right]& \leq &\mathbb{E}\left[\max_{l=0,\dots ,n'-n}\|\bar M_N^l\|_{2,x}^{2q}\right]\\
&\leq& C\sum_{k=n+1}^{n'}\mathbb{E}\big[\|{\bar A}_N^{k}\|_{2,x}^{2q}\big](n'-n)^{q-1}\\
&\leq& C(\|G\|_{2,0}, q)(n'{\Delta t}-n{\Delta t})^q\\
&\leq& C(\|G\|_{2,0}, q)(|t-s|+\Delta t)^q\,.
\end{eqnarray*}
Then, \eqref{kolmogorov_test} follows from \eqref{conv_egoroff} and Fatou's Lemma. Thus, $\bar X$ defines a continous martingale with respect to $(\mathcal{\bar F}_t)_{t\in[0,T]}$ (see \eqref{continuous_filtration}). As we saw in the proof of Proposition \ref{pro_tension} the processes $\bar X_N$, for $N\in\mathbb{N}^*$ are square-integrable, uniformly in $N$, thus the almost sure limit $\bar X$ is square-integrable. This proves Proposition \ref{pro_martingale_barX}.
\end{proof}

We are now ready to apply the continuous martingale representation theorem for Hilbert space-valued Wiener processes. We have shown that the limit process ${\bar X}$ satisfies its hypotheses.
The quadratic variation of ${\bar X}$ is given, for any $a,b\in L^2(D)^3$, by~:
\begin{equation*}
 \big(\ll {\bar X}\gg_ta,b\big)_{2,x}=\int_0^t\big({\bar m(s)}\times(Ga),\,{\bar m(s)}\times(Gb)\big)_{2,x}\ds ,\quad t\in[0,T].
\end{equation*}
There exists an enlarged probability space
$ (\tilde\Omega, \mathcal{\tilde F},\mathbb{\tilde P})$,
with
$\bar\Omega\subseteq\tilde\Omega$,
a filtration
$ \{\mathcal{\tilde F}_t\}$,
and a $L^2(D)^3$-valued Wiener process
$ G\tilde W$ defined on $(\tilde\Omega,\mathcal{\tilde F},\mathbb{\tilde P})$,
such that
${\bar X}$, ${\bar m}$ can be extended to random variables on this space, and
 \begin{equation}\label{representation_martingale}
  \bar X(t,\tilde\omega)=\int_0^t\bar m(s,\tilde\omega)\times{G}d\tilde W(s,\tilde\omega).
 \end{equation} 
 
\section{Identification of the limit}
\label{sec:identification_of_the_limit}

In this section, the purpose is to find a relation between $\bar X$ and the limit $\bar m$. Noticing that $\sum_{0\leq(n+1){\Delta t}\leq t}\big({\bar m}_N^{n+1}-{\bar m}_N^n\big)=\bar m_N(t)-m_0,$ and by definition of ${\bar w}_N^n$ in \eqref{def_bar_w}, one may write $\bar X_N(t)$ as :
\begin{equation}
{\bar X}_N(t)={\bar m}_N(t)-{m}_0 -\sum_{0\leq(n+1){\Delta t}\leq t}{\Delta t}\hspace{0.1cm}{\bar w}_N^n -\sum_{0\leq(n+1){\Delta t}\leq t}\big({\bar m}_N^{n+1}-{\bar m}_N^n-{\bar v}_N^n\big)\label{tautology}
\end{equation} 
for any $t\in[0,T]$.
\begin{proposition}\label{pro_bar_w}
Up to a subsequence~:
\begin{equation*}
\bar w_N\underset{N\to\infty}{\rightharpoonup}\bar w\quad \text{weakly in }L^2(\bar\Omega\times[0,T]\times D)^3\,,
\end{equation*}
with
\begin{equation*}
 \bar w=\Delta\bar m + {\bar m}|\nabla {\bar m}|^2+\bar m\times\Delta\bar m +\frac{1}{2}\Pi_{{\bar m}^\perp}\sum_{i\in\mathbb{N}}({\bar m}\times G_i)\times G_i\,,
\end{equation*}
and $\Pi_{{\bar m(\omega,t,x)}^\perp}$ stands for the $\mathbb{R}^3$ orthogonal projection on ${\bar m(\omega,t,x)}^\perp$, for each $\omega,t,x$\,.
\end{proposition}
\begin{corollary}\label{cor_bar_w}
 Up to a subsequence, for any $t\in[0,T]$,
 $\displaystyle\sum\limits_{0\leq(n+1){\Delta t}\leq t}{\Delta t}\hspace{0.1cm}{\bar w}_N^n$ converges weakly in $L^2(\bar\Omega\times[0,T]\times D)^3$ to
\begin{multline*}
\int_0^t\Bigg(\Delta\bar m(s) + {\bar m}(s)|\nabla {\bar m}(s)|^2+\bar m(s)\times\Delta\bar m(s) \\+\frac{1}{2}\Pi_{{\bar m(s)}^\perp}\sum_{i\in\mathbb{N}}({\bar m}(s)\times G_i)\times G_i\Bigg)\ds\,.
\end{multline*}
\end{corollary}
\begin{proof}[Proof of Proposition \ref{pro_bar_w}.]
Thanks to \eqref{estim_w_2_n}, the equality of the laws of $w_N$ and $\bar w_N$, and Alaoglu theorem, we can assume that up to a subsequence, $\bar w_N$ converges weakly to a limit $\bar w$ in $L^2(\bar\Omega\times[0,T]\times D)^3$.
Because of the strong convergence of $\bar m_N$ to $\bar m$ in $L^2(\bar\Omega\times[0,T]\times D)^3$, one has also~:
\begin{equation}
\label{weak_conv}
(\mathrm{Id}-\bar m_N\times)\bar w_N\underset{N\to\infty}{\rightharpoonup}(\mathrm{Id}-\bar m\times)\bar w\quad \text{weakly in }L^2(\bar\Omega\times[0,T]\times D)^3\,.
\end{equation}
let us first identify the limit of $(\mathrm{Id}-\bar m_N\times)\bar w_N$.\bigskip\\
 \textit{\\Step 1~: let us prove that $(\mathrm{Id}-\bar m\times)\bar w(t,x)\perp {\bar m(t,x)}$ a.s.\ } \\
 By definition of $\bar w_N$,  almost surely, and for almost every $(t,x)\in[0,T]\times D$, one has $ {\bar w}_N(\omega,t,x)\cdot  {\bar m}_N(\omega,t,x)=0$.
Thus for any $\mathbb{R}$-valued test function $\phi\in L^\infty(\bar\Omega\times[0,T]\times D)$, one has
$$\mathbb{E}\left[\int_0^T\int_D({\bar w}_N\cdot{\bar m}_N)\phi\dx\dt\right]=0.$$
On the other hand, by weak convergence of $\bar w_N$ and strong convergence of $\bar m_N\phi$, 
$$ \mathbb{E}\left[\int_0^T\int_D({\bar w}_N\cdot{\bar m}_N)\phi\dx\dt\right] \underset{N\to\infty}{\longrightarrow}\mathbb{E}\left[\int_0^T\int_D({\bar w}\cdot {\bar m})\phi\dx\dt\right]\,.$$
Thus, 
${\bar m}(\omega,t,x)\cdot\bar w(\omega,t,x) = 0,$ for almost all $(\omega,t,x)$, and  $(\mathrm{Id}-\bar m\times)\bar w\perp\bar m$.\bigskip\\
\textit{Step 2~: identification of the limit for specific test functions.}\\
We use the definition of ${\bar w}_N$ \eqref{def_bar_w}, and \eqref{pv}.
Let us take
$$\Phi\in \mathcal{C}\left([0,T];L^\infty(\bar\Omega ; W^{1,\infty}_x))\right)\,,$$ 
and consider a test function of the form
\begin{equation}\label{Phi}
{\bar m}(\omega,t,x)\times\Phi(\omega,t,x)\,,\quad\omega\in\bar \Omega\,,t\in[0,T]\,,x\in D\,.
\end{equation}
We approximate this test function by the sequence of piecewise constant functions $({\bar m}_N\times\Phi_N)$,
where we set for all $N\in\mathbb{N}^*$, all $0\leq n\leq N$, and for $t\in[n{\Delta t},(n+1){\Delta t})$,
$$\Phi_N(\omega,t,x)=\Phi_N^n(\omega,x):=\Phi(\omega,n\Delta t,x)\,.$$
On the one hand, using the strong convergence of $\bar m_N$ to $\bar m$ in $L^2(\bar\Omega\times[0,T]\times D)^3$,  and \eqref{estim_nabla_u_2_n} we have :
 \begin{equation}\label{strong_convergence_m_phi}
 {\bar m}_N\times\Phi_N\underset{N\to\infty}{\longrightarrow}{\bar m}\times\Phi\quad\text{strongly in }L^2(\bar\Omega\times[0,T]\times D)^3\,,
\end{equation}
\begin{equation}\label{bornitude_m_phi}
 \nabla({\bar m}_N\times\Phi_N)\text{ is bounded in }L^2(\bar\Omega\times[0,T]\times D)^{3\times3} \text{ uniformly in }N\,,
\end{equation}
and for any $k\in\{1,2,3\}$,
\begin{equation}\label{strong_convergence_m_Dphi}
 {\bar m}_N\times\partial_{x_k}\Phi_N\underset{N\to\infty}{\longrightarrow}{\bar m}\times\partial_{x_k}\Phi\quad\text{strongly in }L^2(\bar\Omega\times[0,T]\times D)^3\,.
\end{equation}
On the other hand, almost surely, $(\bar m_N\times \Phi_N) \in  \mathbb{W}_{N,n}$, and is therefore a suitable test 
function in the variational formulation \eqref{pv_bar}. Using then \eqref{def_bar_w} and the definition of $\bar A_N^n$, and summing on $n\in\{0,\dots,N-1\}$, one obtains~:
\begin{flalign}
\nonumber\mathbb{E}\Big[\int_0^T&\big( (\mathrm{Id}-\bar m_N\times){\bar w}_N\,,\,{\bar m}_N\times\Phi_N\big)_{2,x}\dt\Big] &\\
\label{identite_particuliere_1}=&-2\theta\mathbb{E}\left[\int_0^T\big(\nabla  {\bar v}_N\,,\,\nabla({\bar m}_N\times\Phi_N)\big)_{2,x}\dt\right]\\
\nonumber&-2\mathbb{E}\left[\int_0^T\big(\nabla {\bar m}_N\,,\,\nabla({\bar m}_N\times\Phi_N)\big)_{2,x}\dt\right]\\
\nonumber&+\frac{1}{2}\mathbb{E}\left[\int_0^T\sum_{i\in\mathbb{N}}\big((\mathrm{Id}-{\bar m}_N\times)\big(({\bar m}_N\times G_i)\times G_i\big)\,,\,{\bar m}_N\times\Phi_N\big)_{2,x}\dt\right]\,.
\end{flalign}

The first term in the right hand side above converges to zero, because of \eqref{estim_v_2_n}, and \eqref{bornitude_m_phi}.
For the second term,
we observe that since for all $k=1,2,3$, 
$$\partial_{x_k}\bar m_N^n\cdot(\partial_{x_k}\bar m_N^n\times \Phi_N^n)=0\,,$$
then
\begin{multline*}
 2\mathbb{E}\left[\int_0^T\big( \nabla {\bar m}_N\,,\,\nabla({\bar m}_N\times\Phi_N)\big)_{2,x}\dt\right]\\ =2\mathbb{E}\left[\int_0^T\big(\sum_{k=1,2,3}\partial_{x_k}{\bar m}_N, {\bar m}_N\times\partial_{x_k}\Phi_N\big)_{2,x}\dt\right]\,.
\end{multline*}
By \eqref{strong_convergence_m_Dphi}, and the weak convergence of $\nabla\bar m_N$ to $\nabla\bar m$ in $L^2(\bar\Omega\times[0,T]\times D)^3$ (see Remark \ref{rema_conv_nabla}) , this tends to $2\mathbb{E}\left[\int_0^T\sum_{k=1,2,3}\big(\partial_{x_k}{\bar m}\,,\, {\bar m}\times\partial_{x_k}\Phi\big)_{2,x}\dt\right]$ as $N\to\infty$\,.

Eventually, it easily follows from assumption \eqref{G_hypothesis}, the Sobolev embedding $H^2(D)^3\subset L^{\infty}(D)^3$, 
and the boundedness of the sequence $(\bar m_N)_N$ in $L^{\infty}(\bar \Omega\times[0,1]\times D)^3$ and \eqref{strong_convergence_m_phi}
that the third term of the right hand side of \eqref{identite_particuliere_1} converges strongly in $L^2(\bar\Omega\times[0,T]\times D)^3$
to 
$$
\frac{1}{2}\sum_{i\in\mathbb{N}}\mathbb{E}\left[\int_0^T\big((\mathrm{Id}-{\bar m}\times)\big(({\bar m}\times G_i)\times G_i\big)\,,\,\bar m\times\Phi\big)_{2,x}\dt\right]\,.
$$

Identifying all the limits in the right hand side of \eqref{identite_particuliere_1}, we get~:
\begin{multline}\label{identite_particuliere_00}
\mathbb{E}\left[\int_0^T\big((\mathrm{Id}-\bar m\times)\bar w,{\bar m}\times\Phi\big)_{2,x}\dt\right] = -2\mathbb{E}\left[\int_0^T\big(\nabla {\bar m},\nabla({\bar m}\times\Phi)\big)_{2,x}\dt\right]\\
+\frac{1}{2}\sum_{i\in\mathbb{N}}\mathbb{E}\left[\int_0^T\big((\mathrm{Id}-{\bar m}\times)\big(({\bar m}\times G_i)\times G_i\big)\,,\,\bar m\times\Phi\big)_{2,x}\dt\right]\,.
\end{multline}
By a density argument,  \eqref{identite_particuliere_00} remains true for any $\Phi\in L^2(\bar\Omega\times[0,T];H^1(D))^3$.\bigskip\\
\textit{Step 3~: identification of the limit for any test function.}\\
We are going to use \eqref{identite_particuliere_00} with $$\Phi:=\bar m\times\Xi\,,$$
where $\Xi\in L^2(\bar\Omega\times[0,T];W^{1,\infty}(D))^3$ 
and thus $\Phi\in L^2(\bar\Omega\times[0,T];H^1(D))^3$.
First, observe that for any unit vector $V\in\mathbb{S}^2$, one has
\begin{equation}\label{id_linear}
 {V}\times({V}\times\cdot)=-\Pi_{{V}^\perp},
\end{equation} 
where $\Pi_{{V}^\perp}$ denotes the orthogonal projection on ${V}^\perp$, hence from Step 1,
\begin{equation}\label{double_prod_1}
 \big((\mathrm{Id}-\bar m\times){\bar w}\big)\cdot\bar m\times({\bar m}\times\Xi)=- \big((\mathrm{Id}-\bar m\times){\bar w}\big)\cdot\Xi\,.
\end{equation}
Moreover, for any $1\leq k\leq3$, since $\bar m\cdot \partial_k\bar m=0$, one has
\begin{align}
\nonumber(\partial_{x_k} {\bar m}\times {\bar m})\cdot\partial_{x_k}({\bar m}\times\Xi)&=(\partial_{x_k} {\bar m}\times {\bar m})\cdot\left((\partial_{x_k}{\bar m}\times\Xi)+({\bar m}\times \partial_{x_k}\Xi)\right) \\
\label{double_prod_2}&=|\partial_{x_k}{\bar m}|^2{\bar m}\cdot\Xi-\partial_{x_k} {\bar m}\cdot\partial_{x_k}\Xi\,.
 \end{align}
 Using \eqref{double_prod_1} and \eqref{double_prod_2} in \eqref{identite_particuliere_00} with $\Phi:=\bar m\times\Xi$, we obtain~:
\begin{flalign}
\nonumber-\mathbb{E}\Big[\int_0^T\big((\mathrm{Id}-&\bar m\times)\bar w,\Xi\big)_{2,x}\dt\Big]\\
\label{identite_particuliere_11}=&2\mathbb{E}\Big[\int_0^T\big(\nabla {\bar m},\nabla\Xi\big)_{2,x}\dt\Big]-2\mathbb{E}\Big[\int_0^T\big(|\nabla {\bar m}|^2{\bar m},\Xi\big)_{2,x}\dt\Big]\\
\nonumber&-\frac{1}{2}\mathbb{E}\Big[\int_0^T\big(\Pi_{{\bar m}^\perp}\Big[\sum_{i\in\mathbb{N}}(\mathrm{Id}-{\bar m}\times)\big(({\bar m}\times G_i)\times G_i\big)\Big],\Xi\big)_{2,x}\dt\Big]\,,
\end{flalign}
from which we deduce that
$$(\mathrm{Id}-\bar m\times )\bar w=2\big(\Delta\bar m +|\nabla\bar m|^2\bar m\big) +\frac{1}{2}\Pi_{{\bar m}^\perp}\Big[\sum_{i\in\mathbb{N}}(\mathrm{Id}-{\bar m}\times)\big(({\bar m}\times G_i)\times G_i\big)\Big]$$
in $L^2(\bar\Omega\times[0,T]\times D)^3$.
\bigskip\\
\textit{Step 4 : end of the proof.\\}
Note that if $V\cdot \bar m=0$, then
$$
(\mathrm{Id}-{\bar m}\times)^{-1}V=\frac12 (\mathrm{Id}+{\bar m}\times)V\,,
$$
so that
\begin{align*}
\bar w=\Delta\bar m+{\bar m}|\nabla {\bar m}|^2+{\bar m}\times\Delta\bar m+\frac{1}{2}\Pi_{{\bar m}^\perp}\sum_{i\in\mathbb{N}}({\bar m}\times G_i)\times G_i\,,
\end{align*}
and Proposition \ref{pro_bar_w} is proved.
\end{proof}
\begin{proof}[Proof of Corollary \ref{cor_bar_w}]
It is an immediate consequence of Proposition \ref{pro_bar_w}, and the fact that
\begin{equation*}
 \mathbb{E}\left[\int_{\left[\frac{t}{\Delta t}\right]\Delta t}^t\big(\bar w_N(s),\Phi\big)_{2,x}\ds\right]\quad \text{tends to }0
\end{equation*}
 as $N\to\infty$, for any $\Phi\in L^2(\bar \Omega \times D)^3$, thanks to \eqref{estim_w_2_n}.
\end{proof}
\begin{proposition}\label{pro_id_bar_X}
For almost every $t\in[0,T]$,
\begin{equation*}
\sum_{0\leq(n+1){\Delta t}\leq t}\big({\bar m}_N^{n+1}-{\bar m}_N^n-{\bar v}_N^n\big)\end{equation*}
converges strongly in $L^1(\bar\Omega\times D)^3$ to 
\begin{equation*}
\frac{1}{2}\int_0^t\Pi_{\bar m}\left[\sum_{i\in\mathbb{N}}(\bar m\times G_i)\times G_i\right]\,.
\end{equation*}
\end{proposition}
\begin{corollary}\label{cor_id_bar_X}
For almost every $t\in[0,T]$,
\begin{equation}\label{identification_X}
 {\bar X}(t)={\bar m}(t)-{m}_0-\int_0^t\big(\Delta {\bar m} + {\bar m}|\nabla m|^2 +{\bar m}\times\Delta{\bar m}  +\frac{1}{2}\sum_{i\in\mathbb{N}}({\bar m}\times G_i)\times G_i\big)\ds\,,
\end{equation}
and $\bar m\in \mathcal{C}([0,T];L^2(D)^3)$.
\end{corollary}
\begin{proof}[Proof of Proposition \ref{pro_id_bar_X}.]
We set for each $0\leq n\leq N$~:
\begin{equation}\label{nota_RN}
R_N^n :={\bar m}_N^{n+1}-{\bar m}_N^n-{\bar v}_N^n+\frac{1}{2}|{\bar A}_N^n|^2{\bar m}_N^n\,.
\end{equation}
It suffices to prove the following two facts :
\begin{equation}\label{convergence_correction_R}
\sum_{0\leq(n+1){\Delta t}\leq t}R_N^n\underset{N\to\infty}{\longrightarrow} 0
\end{equation}
strongly in $L^1(\bar\Omega\times D)^3$, 
and
\begin{equation}\label{convergence_correction_P}
 \sum_{0\leq(n+1){\Delta t}\leq t}|{\bar m}_N^n\times(G\Delta\bar W_N^n)|^2{\bar m}_N^n\underset{N\to\infty}{\longrightarrow}-\int_0^t\Pi_{\bar m}\left[\sum_{i\in\mathbb{N}}({\bar m}\times{G_i})\times G_i\right]\ds \,,
 \end{equation}
strongly in $L^2(\bar\Omega\times D)^3$.
\bigskip\\
\textit{Proof of \eqref{convergence_correction_R}:\\} Let us decompose $ R_N^n$ into four terms, namely
$$
 R_N^n=\frac{{\bar m}_N^{n}+{\bar v}_N^n}{\sqrt{1+|{\bar v}_N^n|^2}}-{\bar m}_N^n-{\bar v}_N^n+\frac{1}{2}|{\bar A}_N^n|^2{\bar m}_N^n
 = I+II+III+IV
$$
with
\begin{eqnarray*}
\textstyle I:={\bar m}_N^n\left(\frac{1}{\sqrt{1+|{\bar A}_N^n|^2}}-1+\frac{1}{2}|{\bar A}_N^n|^2\right), & \; &
\textstyle II:={\bar m}_N^n\left(\frac{1}{\sqrt{1+|{\bar v}_N^n|^2}}-\frac{1}{\sqrt{1+|{\bar A}_N^n|^2}}\right),\\
\textstyle III:={\bar v}_N^n\left(\frac{1}{\sqrt{1+|{\bar A}_N^n|^2}}-1\right) \mbox{ and } & \; &
\textstyle IV:={\bar v}_N^n\left(\frac{1}{\sqrt{1+|{\bar v}_N^n|^2}}-\frac{1}{\sqrt{1+|{\bar A}_N^n|^2}}\right)\,,
\end{eqnarray*} 
and treat each of them separately.

\noindent
Convergence of $I$: Using 
$$
\big| \frac{1}{\sqrt{1+x^2}}-(1-\frac12 x^2)\big| \leq Cx^4, \; \mbox{for all} \; x \in \mathbb{R},
$$
we get
\begin{flalign*}
\mathbb{E}\Big[\sum_{0\leq (n+1){\Delta t}\leq t}\Big\|{\bar m}_N^n\Big(\frac{1}{\sqrt{1+|{\bar A}_N^n|^2}}-1+\frac{1}{2}&|{\bar A}_N^n|^2\Big)\Big\|_{1,x}\Big]&\\
&\leq \frac{1}{2}\mathbb{E}\Big[\sum_{0\leq (n+1){\Delta t}\leq t} \|{\bar A}_N^n\|_{4,x}^4\Big]&\\
&\leq C\|G\|_{2,1}^4\sum_{0\leq (n+1){\Delta t}\leq t} {\Delta t}^2\,,&
\end{flalign*}
which tends to zero as $N$ tends to infinity.

\noindent
Convergence of $II$:
we have by Cauchy-Schwarz inequality and a H\"older-type inequality~:
\begin{eqnarray*}
& & \mathbb{E}\Big[\sum_{0\leq (n+1){\Delta t}\leq t} \textstyle\Big\| {\bar m}_N^n(\frac{1}{\sqrt{1+|{\bar v}_N^n|^2}}-\frac{1}{\sqrt{1+|{\bar A}_N^n|^2}})\Big\|_{1,x}\Big]\\
&& \hspace{1cm} \leq C\mathbb{E}\Big[\sum_{0\leq (n+1){\Delta t}\leq t}\left\||{\bar v}_N^n|^2-|{\bar A}_N^n|^2\right\|_{1,x}\Big]\\
&& \hspace{1cm}\leq C\mathbb{E}\Big[\sum_{0\leq (n+1){\Delta t}\leq t}\left\|{\bar v}_N^n+{\bar A}_N^n\right\|_{2,x}\left\|{\bar v}_N^n-{\bar A}_N^n\right\|_{2,x}\Big]\\
&& \hspace{1cm}\leq \frac{C}{2\sqrt{\Delta t}}\mathbb{E}\Big[\sum_{0\leq (n+1){\Delta t}\leq t}\|{{\bar v}_N^n-{\bar A}_N^n}\|_{2,x}^2\Big]\\
&&\hspace{1.3cm}+ \frac{C\sqrt{\Delta t}}{2}\mathbb{E}\Big[\sum_{0\leq (n+1){\Delta t}\leq t}\|{{\bar v}_N^n+{\bar A}_N^n}\|_{2,x}^2\Big]\,.
\end{eqnarray*}
We then use \eqref{estim_A_2}, \eqref{estim_w} and \eqref{estim_v} to conclude that $II$ tends to $0$ as $N$ tends to infinity.

\noindent
Convergence of $III$~:
Working as above, one has
\begin{multline*}
\mathbb{E}\Big[\sum_{0\leq(n+1){\Delta t}\leq t}\textstyle \Big\|{\bar v}_N^n\big(\frac{1}{\sqrt{1+|{\bar A}_N^n|^2}}-1\big)\Big\|_{1,x}\Big] \\
\leq\frac{\sqrt{\Delta t}}{2}\Big(\sum_{0\leq(n+1){\Delta t}\leq t}\|{\bar v}_N^n\|_{2,x}^2\Big)+\frac{C}{2\sqrt{\Delta t}}\Big(\sum_{0\leq(n+1){\Delta t}\leq t}\|{\bar A}_N^n\|_{4,x}^4\Big)\,.
\end{multline*}
Then, using again \eqref{estim_v}, and \eqref{estim_A_4}, the above quantity tends to $0$ as $N\to\infty$.

\noindent
Convergence of $IV$~:
Cauchy-Schwarz inequality implies~:
\begin{multline*}
\textstyle \mathbb{E}\Big[\big\|{\bar v}_N^n\big(\frac{1}{\sqrt{1+|{\bar v}_N^n|^2}}-\frac{1}{\sqrt{1+|{\bar A}_N^n|^2}}\big)\big\|_{1,x}\Big]\\
\leq\Big(\mathbb{E}\big[\sum_{0\leq(n+1){\Delta t}\leq t}\|{\bar v}_N^n\|_{2,x}^2\big]\Big)^{\frac{1}{2}}\Big(\mathbb{E}\big[\sum_{0\leq(n+1){\Delta t}\leq t}\|{\bar v}_N^n-{\bar A}_N^n\|_{2,x}^2\big]\Big)^{\frac{1}{2}}\,.
\end{multline*}
Using \eqref{estim_w} and \eqref{estim_v}, we conclude that $IV$ tends to $0$ as $N\to\infty$.
Finally \eqref{convergence_correction_R} is proved.\bigskip\\
\textit{Proof of \eqref{convergence_correction_P}.}\\

For each $N\in\mathbb{N}^*$, and $0\leq n\leq N$, observe that if we denote by $\mathbb{E}[\cdot|\mathbb{\bar F}_N^n]$ the conditional expectation with respect to $\mathbb{\bar F}_N^n$ (see \eqref{discrete_filtration_bar}), we have
\begin{equation}\label{bar_m_times_G_2}
 \mathbb{E}\left[|\bar m_N^n\times(G\Delta\bar W_N^n)|^2\big|\mathbb{F}_N^n\right]=\Delta t\sum_{i\in\mathbb{N}}|\bar m_N^n\times G_i|^2\,.
\end{equation}
We set
\begin{equation*}
 \sum_{0\leq(n+1){\Delta t}\leq t}{\bar m}_N^n|{\bar m}_N^n\times(G\Delta\bar W_N^n)|^2-\int_0^t \sum_{i\in\mathbb{N}}{\bar m}|{\bar m}\times{G_i}|^2\ds =I+II\,,
\end{equation*}
where
\begin{equation*}
I=\sum_{0\leq(n+1){\Delta t}\leq t}{\bar m}_N^n\Big(|{\bar m}_N^n\times(G\Delta\bar W_N^n)|^2-\sum_{i\in\mathbb{N}}|{\bar m}_N^n\times G_i|^2{\Delta t}\Big)\,,
\end{equation*}
and
\begin{equation*}
II=\sum_{0\leq(n+1){\Delta t}\leq t} \sum_{i\in\mathbb{N}} {\bar m}_N^n|{\bar m}_N^n\times{G_i}|^2{\Delta t}-\int_0^t\sum_{i\in\mathbb{N}}{\bar m}|{\bar m}\times{G_i}|^2\ds\,.
\end{equation*}

Let us prove that the first term above tends to zero as $N$ tends to infinity in $L^2(\bar\Omega\times D)^3$.
\begin{equation*}
\mathbb{E}[\|I\|_{2,x}^2]=\mathbb{E}\Big[\big\|\sum_{0\leq(n+1){\Delta t}\leq t}{\bar m}_N^n\big(|{\bar m}_N^n\times(G\Delta\bar W_N^n)|^2-\sum_{i\in\mathbb{N}}|{\bar m}_N^n\times G_i|^2{\Delta t}\big)\big\|_{2,x}^2\Big]
\end{equation*}
Developing the square under the expectation above, we get a sum over two indices $k\leq n$, which contains the following terms~:
\begin{multline*}
\mathbb{E}\Big[\big({\bar m}_N^n(|{\bar m}_N^n\times(G\Delta\bar W_N^n)|^2-\sum_{i\in\mathbb{N}}|{\bar m}_N^n\times G_i|^2{\Delta t})\,,\\
{\bar m}_N^k(|{\bar m}_N^k\times{G\Delta{\bar W}_N^k}|^2-\sum_{i\in\mathbb{N}}|{\bar m}_N^k\times G_i|^2{\Delta t})\big)_{2,x}\Big]\,.
\end{multline*}
When $k<n$, this is equal to 
\begin{multline*}
\mathbb{E}\Big[\big({\bar m}_N^n\mathbb{E}\big[ |{\bar m}_N^n\times(G\Delta\bar W_N^n)|^2-\sum_{i\in\mathbb{N}}|{\bar m}_N^n\times G_i|^2{\Delta t}\big|\bar{\mathbb F}_N^n\big]\,,\\
{\bar m}_N^k(|{\bar m}_N^k\times{G\Delta{\bar W}_N^k}|^2-\sum_{i\in\mathbb{N}}|{\bar m}_N^k\times G_i|^2{\Delta t})\big)_{2,x}\Big]\,.
\end{multline*}
Thus, using \eqref{bar_m_times_G_2}, these terms vanish. It follows, using again \eqref{bar_m_times_G_2}
\begin{eqnarray*}
\mathbb{E}[\|I\|_{2,x}^2]&=&\sum_{0\leq(n+1)\Delta t \leq t}\mathbb{E}\Big[\big\||{\bar m}_N^n\times(G\Delta\bar W_N^n)|^2-\sum_{i\in\mathbb{N}}|{\bar m}_N^n\times G_i|^2{\Delta t}\big\|_{2,x}^2\Big]\\
&=&\sum_{0\leq(n+1)\Delta t\leq t}\mathbb{E}\Big[\int\limits_D|{\bar m}_N^n\times(G\Delta\bar W_N^n)|^4-\big(\sum_{i\in\mathbb{N}}|{\bar m}_N^n\times G_i|^2\big)^2{\Delta t^2}\dx\Big]\,.
\end{eqnarray*}
Both terms on the right hand side are bounded by $C(T)\|G\|_{2,1}^4\Delta t$ (see \eqref{estim_A_4}), hence tend to zero as $N$ tends to infinity.

For $II$, we write
\begin{equation*}
\sum_{0\leq(n+1){\Delta t}\leq t}\sum_{i\in\mathbb{N}}{\bar m}_N^n|{\bar m}_N^n\times{G_i}|^2{\Delta t}=\int_0^{[\frac{t}{{\Delta t}}]{\Delta t}}\sum_{i\in\mathbb{N}}{\bar m}_N| {\bar m}_N\times{G_i}|^2\ds \,,
\end{equation*}
and note that $\int_{{[\frac{t}{{\Delta t}}]{\Delta t}}}^t\sum_{i\in\mathbb{N}}{\bar m}_N| {\bar m}_N\times{G_i}|^2\ds$
tends to zero in $L^2(\bar\Omega\times D)^3$.
Moreover, we have that
\begin{multline}
\mathbb{E}\left[\big\| \int_0^{t} \sum_{i\in\mathbb{N}}{\bar m}_N| {\bar m}_N\times{G_i}|^2\ds -\int_0^t \sum_{i\in\mathbb{N}}{\bar m}|{\bar m}\times{G_i}|^2\ds \big\|_{2,x}\right] \\ \leq C\|G\|_{2,2}^2\mathbb{E}\int_0^T\|\bar m_N-\bar m\|_{2,x}\,\ds
\end{multline}
tends to 0 as $N$ tends to infinity.
We conclude that
$$\displaystyle\sum_{0\leq(n+1){\Delta t}\leq t}\sum_{i\in\mathbb{N}} {\bar m}_N^n | {\bar m}_N^n\times{G_i}|^2{\Delta t}\underset{N\to\infty}{\longrightarrow}\int_0^t \sum_{i\in\mathbb{N}}{\bar m}|{\bar m}\times{G_i}|^2\ds$$
strongly in $L^1(\bar\Omega;L^2( D)^3)$. Eventually,  \eqref{convergence_correction_P} follows from
$$\int_0^t \sum_{i\in\mathbb{N}}{\bar m}|{\bar m}\times{G_i}|^2\ds=-\int_0^t \Pi_{{\bar m}}\sum_{i\in\mathbb{N}}(({\bar m}\times G_i)\times G_i)\ds\,. $$
\end{proof}
\begin{proof}[Proof of Corollary \ref{cor_id_bar_X}.]
We recall that we have for almost every $t\in[0,T]$ and for all $N\in\mathbb{N}^*$
\begin{equation*}
 \bar X_N(t)=\bar m_N(t)-{m}_0-\sum_{0\leq(n+1){\Delta t}\leq t}{\Delta t}{\bar w}_N^n-\sum_{0\leq(n+1){\Delta t}\leq t}({\bar m}_N^{n+1}-{\bar m}_N^n-{\bar v}_N^n)
\end{equation*}
and that, up to the extraction of a subsequence, 
\begin{equation*}
  {\bar m}_N\longrightarrow {\bar m}\quad\text{ strongly in }L^2(\bar\Omega\times[0,T]\times D)^3,
\end{equation*}
and thus without loss of generality, we can assume that for almost any $t\in[0,T]$,
\begin{equation*}
  {\bar m}_N(t)\longrightarrow {\bar m}(t) \quad L^2(\bar\Omega\times D)^3\,.
\end{equation*} 
Then, thanks to Corollary \eqref{cor_bar_w} and Proposition \ref{pro_id_bar_X}, 
\begin{equation*}
-\sum_{0\leq(n+1){\Delta t}\leq t}{\Delta t}{\bar w}_N^n-\sum_{0\leq(n+1){\Delta t}\leq t}({\bar m}_N^{n+1}-{\bar m}_N^n-{\bar v}_N^n)
\end{equation*}
converges weakly in $L^2(\bar\Omega\times D)^3$ to 
\begin{multline*}
 -\int_0^t\left( \Delta\bar m + \bar m|\nabla\bar m|^2+\bar m\times\Delta\bar m+\frac12\Pi_{\bar m^\perp}\left[\sum_{i\in\mathbb{N}}(\bar m \times G_i)\times G_i\right]\right)\ds\\
 - \frac12\int_0^t\left(\Pi_{\bar m}\left[\sum_{i\in\mathbb{N}}(\bar m \times G_i)\times G_i\right]\right)\ds
\end{multline*}
The continuity of $\bar m$ follows from \eqref{identification_X}, the continuity of $\bar X$ (Proposition \ref{pro_martingale_barX}), and the fact that $\bar{w}\in L^1(0,T;L^2(D)^3)$ a.s. (see Proposition \ref{pro_bar_w}).
\end{proof}
\noindent {\em End of proof of Theorem \ref{theo_main_result}. }The conclusion follows, thanks to \eqref{identification_X}, together with \eqref{representation_martingale}~:
there exists a martingale solution $(\tilde\Omega,\mathbb{\tilde P},\mathcal{\tilde F}_{t\in[0,T]},\tilde W,\bar m)$ of \eqref{LLG_Ito}, i.e.\
for any time $t\in[0,T]$,
\begin{multline*}
 {\bar m}(t)={m}_0+\int_0^t\left(\Delta\bar m + {\bar m}|\nabla {\bar m}|^2 +\bar m\times\Delta m +\frac{1}{2}\sum_{i\in\mathbb{N}}({\bar m}\times{G_i})\times{G_i}\right)\ds\\  +\int_0^t {\bar m}\times( G\ddW(s))\,.
\end{multline*}
\qed

\paragraph{Acknowledgements :}Partial funding of this research through the ANR projects Micro-MANIP (ANR-08-BLAN-0199) and STOSYMAP (ANR-2011-BS01-015-03) is gratefully acknowledged.
\bibliographystyle{abbrv}

\end{document}